\documentclass[final,12pt]{elsarticle}
\usepackage{srcltx}
\usepackage{eurosym}
\usepackage{mathtools}
\usepackage{amsmath}
\usepackage{amsfonts}
\usepackage{amssymb}
\usepackage{amsthm}
\usepackage{graphicx}
\usepackage{mathrsfs}
\usepackage{xcolor}
\usepackage{exscale}
\usepackage{latexsym}
\usepackage{showkeys}

\numberwithin{equation}{section}

\usepackage[colorlinks,plainpages=true,pdfpagelabels,hypertexnames=true,colorlinks=true,pdfstartview=FitV,linkcolor=blue,citecolor=red,urlcolor=black]{hyperref}
\PassOptionsToPackage{unicode}{hyperref}
\PassOptionsToPackage{naturalnames}{hyperref}
\usepackage{enumerate}
\usepackage[shortlabels]{enumitem}
\usepackage{bookmark}
\usepackage{wasysym}
\usepackage{esint}
\usepackage[ddmmyyyy]{datetime}
\usepackage[margin=2cm]{geometry}
\makeatletter
\g@addto@macro\normalsize{%
  \setlength\abovedisplayskip{4pt}
  \setlength\belowdisplayskip{4pt}
  \setlength\abovedisplayshortskip{4pt}
  \setlength\belowdisplayshortskip{4pt}
}
\everymath{\displaystyle}
\usepackage[capitalize,nameinlink]{cleveref}
\crefname{section}{Section}{Sections}
\crefname{subsection}{Subsection}{Subsections}
\crefname{condition}{Condition}{Conditions}
\crefname{hypothesis}{Hypothesis}{Conditions}
\crefname{assumption}{Assumption}{Assumptions}
\crefname{lemma}{Lemma}{Lemmas}
\crefname{definition}{Definition}{Definitions}

\crefformat{equation}{\textup{#2(#1)#3}}
\crefrangeformat{equation}{\textup{#3(#1)#4--#5(#2)#6}}
\crefmultiformat{equation}{\textup{#2(#1)#3}}{ and \textup{#2(#1)#3}}
{, \textup{#2(#1)#3}}{, and \textup{#2(#1)#3}}
\crefrangemultiformat{equation}{\textup{#3(#1)#4--#5(#2)#6}}%
{ and \textup{#3(#1)#4--#5(#2)#6}}{, \textup{#3(#1)#4--#5(#2)#6}}%
{, and \textup{#3(#1)#4--#5(#2)#6}}

\Crefformat{equation}{#2Equation~\textup{(#1)}#3}
\Crefrangeformat{equation}{Equations~\textup{#3(#1)#4--#5(#2)#6}}
\Crefmultiformat{equation}{Equations~\textup{#2(#1)#3}}{ and \textup{#2(#1)#3}}
{, \textup{#2(#1)#3}}{, and \textup{#2(#1)#3}}
\Crefrangemultiformat{equation}{Equations~\textup{#3(#1)#4--#5(#2)#6}}%
{ and \textup{#3(#1)#4--#5(#2)#6}}{, \textup{#3(#1)#4--#5(#2)#6}}%
{, and \textup{#3(#1)#4--#5(#2)#6}}

\crefdefaultlabelformat{#2\textup{#1}#3}

\newtheorem{theorem} {Theorem}[section]
\newtheorem{proposition}[theorem]{Proposition}
\newtheorem{lemma}[theorem]{Lemma}
\newtheorem{corollary}[theorem]{Corollary}

\newtheorem{counter example}[theorem]{Counter Example}
\newtheorem{remark}[theorem] {Remark}
\newtheorem{definition}[theorem] {Definition}

\newtheorem{assumption}[theorem]{Assumption}

\def\N{\mathbb{N}}
\def\CC{{\rm \kern.24em \vrule width.02em height1.4ex depth-.05ex \kern-.26emC}}

\def\TagOnRight

\def\AA{{it I} \hskip-3pt{\tt A}}

\def\QQ{\rlap {\raise 0.4ex \hbox{$\scriptscriptstyle |$}} {\hskip -0.1em Q}}

\makeatletter
\newcommand{\vo}{\vec{o}\@ifnextchar{^}{\,}{}}
\makeatother

\def\YYint#1#2#3{{\setbox0=\hbox{$#1{#2#3}{\iint}$}
    \vcenter{\hbox{$#2#3$}}\kern-.50\wd0}}

\def\XXint#1#2#3{{\setbox0=\hbox{$#1{#2#3}{\int}$}
    \vcenter{\hbox{$#2#3$}}\kern-.50\wd0}}

\makeatletter
\def\namedlabel#1#2{\begingroup
   \def\@currentlabel{#2}%
   \label{#1}\endgroup
}
\makeatother
\makeatletter
\newcommand{\rmh}[1]{\mathpalette{\raisem@th{#1}}}
\newcommand{\raisem@th}[3]{\hspace*{-1pt}\raisebox{#1}{$#2#3$}}
\makeatother

\newcounter{desccount}
\newcommand{\descitem}[2]{\item[#1]\refstepcounter{desccount}\label{#2}}
\newcommand{\descref}[2]{\hyperref[#1]{\textcolor{black}{}\textcolor{blue}{ #2}\textcolor{black}{}}}

\newcommand{\dref}[2]{\hyperref[#1]{\textcolor{black}{(}\textcolor{blue}{\bf #2}\textcolor{black}{)}}}
\newcommand{\be} {\begin{eqnarray}}
\newcommand{\ee} {\end{eqnarray}}
\newcommand{\Bea} {\begin{eqnarray*}}
\newcommand{\Eea} {\end{eqnarray*}}
\newcommand{\pa} {\partial}
\newcommand{\re}{\mathbb{R}}

 \newcommand{\al} {\alpha}
\newcommand{\rr}{\rightarrow}

\newcommand{\B} {\beta}
\newcommand{\de} {\delta}
\newcommand{\g} {\gamma}

\newcommand{\p}  {\prime}
\newcommand{\e}  {\epsilon}
\newcommand{\De} {\Delta}

\newcommand{\la} {\lambda}
\newcommand{\si} {\sigma}

\newcommand{\f}{\infty}
\newcommand{\R}{\mathbb{R}}
\newcommand{\eps} {\epsilon}

\newcommand{\ga}{\gamma}

\newcommand{\Dom}{Q}

\DeclareMathOperator{\dv}{div}

\newcommand{\abs}[1]{\left| #1\right|}

\newcounter{whitney}
\refstepcounter{whitney}

\newcounter{ineqcounter}
\refstepcounter{ineqcounter}
\makeatletter
\def\ps@pprintTitle{%
\let\@oddhead\@empty
\let\@evenhead\@empty
\def\@oddfoot{}%
\let\@evenfoot\@oddfoot}
\makeatother
\usepackage[doublespacing]{setspace}
\usepackage[titletoc,toc,page]{appendix}

\makeatletter
\newcommand{\refcheckize}[1]{%
  \expandafter\let\csname @@\string#1\endcsname#1%
  \expandafter\DeclareRobustCommand\csname relax\string#1\endcsname[1]{%
    \csname @@\string#1\endcsname{##1}\wrtusdrf{##1}}%
  \expandafter\let\expandafter#1\csname relax\string#1\endcsname
}
\makeatother

\refcheckize{\cref}
\refcheckize{\Cref}


\makeatletter
\newcommand{\mainsectionstyle}{%
	\renewcommand{\@secnumfont}{\bfseries}
	\renewcommand\section{\@startsection{section}{2}%
		\z@{.5\linespacing\@plus.7\linespacing}{-.5em}%
		{\normalfont\bfseries}}%
}
\makeatother
\usepackage{pgf,tikz}
\usetikzlibrary{arrows}
\usetikzlibrary{decorations.pathreplacing}

\usepackage{xpatch}
\xpatchcmd{\MaketitleBox}{\hrule}{}{}{}
\xpatchcmd{\MaketitleBox}{\hrule}{}{}{}

\date{}
\begin{document}

		\begin{frontmatter}
			\title{On the uniqueness of solutions to hyperbolic systems of conservation laws}
	
			\author[myaddress1]{Shyam Sundar Ghoshal}\ead{ghoshal@tifrbng.res.in}
			\author[myaddress1]{Animesh Jana}\ead{animesh@tifrbng.res.in}
			\author[myaddress3]{Konstantinos Koumatos}\ead{K.Koumatos@sussex.ac.uk}
			\address[myaddress1]	{Tata Institute of Fundamental Research,Centre For Applicable Mathematics,
				Sharada Nagar, Chikkabommsandra, Bangalore 560065, India.}
			\address[myaddress3]{Department of Mathematics, University of Sussex, Pevensey 2 Building, Falmer,
				Brighton, BN1 9QH, UK.}			
	\begin{abstract}
	 For general hyperbolic systems of conservation laws we show that dissipative weak solutions belonging to an appropriate Besov space $B^{\alpha,\f}_q$ and satisfying a one-sided bound condition are unique within the class of dissipative solutions. {The exponent $\alpha>1/2$ is universal independently of the nature of the nonlinearity and the Besov regularity need only be imposed in space when the system is expressed in appropriate variables.} The proof utilises a commutator estimate which allows for an extension of the relative entropy method to the required regularity setting. The systems of elasticity, shallow water magnetohydrodynamics, and isentropic Euler are investigated, recovering recent results for the latter. Moreover, the article explores a triangular system motivated by studies in chromatography and constructs an explicit solution which fails to be Lipschitz, yet satisfies the conditions of the presented uniqueness result.
\end{abstract}
		\end{frontmatter}
	\tableofcontents	
	\section{Introduction}
	For $\Omega\subset\R^d$ and $T>0$ arbitrary, consider the system of conservation laws
	\begin{equation}\label{eqn:conlaw}
	\pa_t U(x,t)+\sum\limits_{k=1}^{d}\pa_{k}f_k(U(x,t))=0 \mbox{ for }(x,t)\in \Omega\times[0,T],
	\end{equation}
	to be solved for the unknown function $U:\re^d\times\re_+\rr\re^m$, where $f_k:\re^m\rr\re^m$ are sufficiently smooth, constituent functions and $\partial_k = \partial/\partial x_k$. To avoid technical difficulties, we henceforth consider $\Omega = Q = [0,1]^d$, the $d$-dimensional unit torus, although this is not a restriction. 
	
	Moreover, we assume that system \eqref{eqn:conlaw} is supplemented by an inequality of the form
	\begin{equation}\label{eq:entropy_ineq}
	\pa_t \eta(U) + \pa_k q_k(U) \leq 0,
	\end{equation}
	where we have employed the Einstein summation convention. The function $\eta:  \re^m\rr\re$ is referred to as the entropy and $q=(q_1,\cdots,q_d):\re^m\rr\re^d$ as the entropy flux and it is assumed that they are related to the fluxes $f_k$ by
	 \begin{equation}
	D q_k(U) = D\eta(U)^{T}D f_k(U).\label{eq:entropy-entropy_flux}
	\end{equation}
Inequality \eqref{eq:entropy_ineq} is the Clausius-Duhem inequality and it expresses the second law of thermodynamics in the context of continuum mechanics. Note that any Lipschitz solution to \eqref{eqn:conlaw} satisfies the companion conservation law \eqref{eq:entropy_ineq} as an equality. 
		
	Entropies in physical systems are often convex and in this article we assume this to be the case. In fact, \eqref{eq:entropy-entropy_flux} implies
	\begin{equation}
	D^2\eta(U)Df_k(U) = Df_k(U)^T D^2\eta(U)\mbox{ for }k=1,\cdots,d,\label{eq:symmetric}
	\end{equation}
rendering \eqref{eqn:conlaw} a symmetrisable hyperbolic system upon the change of variables $U\mapsto D\eta(U)$ and thus locally well-posed, see \cite[section 3.2]{Daf_book}. That is, for initial data of sufficiently high regularity, there exists a unique \emph{strong} solution to \eqref{eqn:conlaw}, satisfying \eqref{eq:entropy_ineq} as an equality, on a generally finite time interval. We refer the reader to \cite{BB,Bressan-book} for global existence results in one spatial dimension.

{\color{black} More generally, we may consider systems of the form
\begin{equation}\label{eqn:conlawgen}
\pa_t A(U(x,t))+ \pa_{k}F_k(U(x,t))=0 \mbox{ for }(x,t)\in \Dom\times[0,T],
\end{equation}
for smooth mappings $A,\,F_k:\mathcal{O}\subset \re^m\rr\re^m$ where $\mathcal{O}$ is an open, convex set containing the range of admissible maps $U$ and $DA$ is nonsingular on $\mathcal{O}$. For system \eqref{eqn:conlawgen} we also assume the existence of an entropy-entropy flux pair $(H,Q_k)$, $k=1,\cdots, d$, meaning that there exists $G:\mathcal{O}\rr\re^m$ such that
 \begin{equation}\label{eq:entropy-entropy_flux gen}
D H(U) = G(U) DA(U),\mbox{ and }D Q_k(U) = G(U) D F_k(U).
\end{equation}
The above relations imply that
\begin{equation}\label{eq:symmetricgen}
D G(U)^T DA(U) = DA(U)^TDG(U)\mbox{ and }DG(U)^TDF_k(U) = DF_k(U)^TDG(U)
\end{equation}
and inequality \eqref{eq:entropy_ineq} now becomes
\begin{equation}\label{eq:entropy_ineq gen}
\pa_t H(U) + \pa_k Q_k(U) \leq 0.
\end{equation}
Moreover, convexity of $\eta$ is now replaced by the assumption $DG^T(U)DA(U)>0$, i.e. that the symmetric matrix $DG^T(U)DA(U)$ is positive-definite which, by \eqref{eq:entropy-entropy_flux gen}, is equivalent to
\begin{equation}\label{eq:convexity gen}
D^2H(U) - G(U)D^2A(U) > 0 \mbox{ for all }U\in\mathcal{O}.
\end{equation}
Note that we may recover system \eqref{eqn:conlaw} by setting $V = A(U)$, $f_k = F_k \circ A^{-1}$, $\eta = H\circ A^{-1}$ and $q_k = Q_k\circ A^{-1}$. Then, we formally compute that
\begin{align*}
\pa_t V + \pa_k f_k(V) & = 0\\
\pa_t \eta(V) + \pa_k q_k(V) & = 0.
\end{align*}
Moreover, \eqref{eq:entropy-entropy_flux gen} becomes equivalent to \eqref{eq:entropy-entropy_flux}, i.e.
\[
D q_k(U) = D\eta(U)^{T}D f_k(U)
\]
where $G(U) = D\eta(A(U))$ and the entropy $\eta$ satisfies $D^2\eta(V) > 0$ for $V\in A(\mathcal{O})$.} 
Systems of the form \eqref{eqn:conlawgen} are typical in continuum mechanics where the function $U$ may represent mass, momentum, energy, and other relevant quantities. The reader is referred to \cite{tzavaras_cleopatra,Daf_book,GKS} for examples, as well as Sections \ref{models}, \ref{sec:triangular} where we apply our result to model equations.

However, solutions to hyperbolic systems typically develop singularities in finite time, even if they emanate from smooth initial data, and thus weaker forms of solutions are sought, such as weak or measure-valued solutions \cite{Diperna_measure-valued}. Then, inequality \eqref{eq:entropy_ineq gen} is expected to serve as an admissibility criterion, singling out physically relevant solutions. {\color{black}Of course, for $m>1$, the existence of $H$, $Q$ satisfying \eqref{eq:entropy-entropy_flux gen} is not trivial. For some systems like the Euler equations or hyperelasticity, the energy of the system plays the role of the entropy and it is the dissipation of energy, i.e. $d\eta/dt  \leq 0$, that is often regarded as an admissibility criterion. We will do so here and make this precise in the following section. }

A natural question then arises regarding the uniqueness of these weaker notions of solutions under appropriate entropy-related admissibility criteria. For example, systems of the form \eqref{eqn:conlaw} endowed with a convex entropy, enjoy a weak-strong uniqueness property whereby any Lipschitz solution (referred to as strong) is unique within the class of (dissipative) weak solutions \cite{Daf_book}, see Definition \ref{def:weak} for the notion of dissipative solution. Similarly, the weak-strong uniqueness result can be extended to systems of the form \eqref{eqn:conlawgen} endowed with an entropy satisfying \eqref{eq:convexity gen}, see \cite{tzavaras_cleopatra}. These weak-strong uniqueness results are based on the relative entropy method, introduced by Dafermos \cite{Daf79} and DiPerna \cite{DiPerna}, which provides a way to estimate the difference between two solutions. The technique has been applied successfully to a number of problems, including extensions of weak-strong uniqueness results to measure-valued solutions \cite{BDS,GKS,Wid18}, convergence of discrete schemes to smooth solutions \cite{tzavaras,westdickenberg}, or relaxation problems \cite{LaTz}

Crucially however the method as originally presented relies on two facts: (a) the system must be endowed with a strictly convex entropy (respectively an entropy satisfying \eqref{eq:convexity gen}) and (b) one of the two solutions needs to enjoy Lipschitz regularity. Relaxing any of these assumptions is of relevance to physical problems and several extensions exist in the literature. For example, in relaxing convexity, the reader is referred to \cite{tzavaras-weak-strong,KS19} in the context of poly- or quasi-convex elasticity, or \cite{Daf86, KV} for conservation laws with involutions. 

In the present article, we focus on relaxing the latter assumption, that is the Lipschitz regularity of the strong solution. In the context of fluid dynamics, the question of uniqueness of shock-free solutions for the Riemann problem to the Euler system has been studied extensively \cite{CF01,ChFrLi,FeKeVa} and in 1-D the relative entropy method has been extended to prove uniqueness of shock wave solutions within a certain class of bounded solutions satisfying a trace property \cite{vasseur}. However, in higher dimensions and based on the theory of convex integration, introduced in this context by DeLellis and Sz{\'e}kelyhidi \cite{DeLS}, uniqueness seems to fail. Indeed, uniqueness fails even for solutions satisfying an energy inequality \cite{CDK,Chiodaroli}. It is important to note that these latter solutions, constructed in \cite{CDK,Chiodaroli}, emanate from planar Riemann data (one dimensional Riemann data extended as constants in the other dimension) containing at least one shock if seen as 1-D data. 

On the contrary, rarefaction solutions to the Riemann problem for compressible Euler remain unique in the class of bounded entropy solutions \cite{FeiKre}. More generally, it was shown recently that, for isentropic Euler, dissipative weak solutions enjoying a certain Besov regularity, and a one-sided Lipschitz condition on the velocity gradient, are unique within the class of weak solutions \cite{FGJ}. {\color{black}We note that the Besov regularity need only be assumed for $t\geq\de$ for every $\de>0$ and thus allows for discontinuous initial data.} This uniqueness result is also achieved via the relative entropy method combined with an appropriate commutator estimate which forces terms produced by regularising the Besov solution to vanish. Commutator estimates have been widespread in the modern literature of conservation laws including \cite{CET} for the Onsager conjecture on energy conservation for the incompressible Euler system, or \cite{FGGW} for compressible Euler, {\color{black} and \cite{GJ} for uniqueness results in the spirit of \cite{FGJ} for compressible Euler. }

In the present article, we employ an appropriate commutator estimate and extend the results of \cite{FGJ} and \cite{GJ} to general systems of conservation laws as in \eqref{eqn:conlawgen} satisfying the symmetrisability condition \eqref{eq:convexity gen} and certain mild assumptions on the functions $A$, $F_k$, $G$, and $H$. In particular, in Theorem \ref{theorem1}, we prove that bounded, dissipative solutions in the Besov space $B^{\alpha,\infty}_q$, $\alpha>1/2$, satisfying a certain one-sided bound condition, see \eqref{ineq:general}, are unique within the class of dissipative solutions. We stress the important 
fact that the exponent, $\alpha>1/2$, in the assumed Besov regularity is universal for general systems of the form \eqref{eqn:conlawgen} and that, in the case of system \eqref{eqn:conlaw}, we can prove our result without assuming any Besov regularity in time. This expands the set of solutions with the uniqueness property and becomes relevant in applications, see \cite{stochastic}. 

As an application of our general theorem, we investigate the isentropic Euler system - recovering the results of \cite{FGJ} - but also the system of conservation laws appearing in polyconvex elasticity and swallow water magnetohydrodynamics, examining the one-sided condition \eqref{ineq:general} in these systems. In discussing polyconvex elasticity, we first consider the system of elasticity under a convexity assumption and comment on the better understood one-dimensional case. As a further, nontrivial example we also explore a one-dimensional triangular system motivated by multi-component chromatography where, for any $\al\in (0,1)$, we construct a solution which lies in the H\"older space $C^{0,\al}$, yet is not Lipschitz, and satisfies the one-sided condition ensuring uniqueness. The construction is then extended to the multi-dimensional setting. Note that such nontrivial examples are lacking in the other systems examined.

The article is organised as follows: in Section \ref{section:prelims} we introduce the necessary terminology, we make our assumptions precise and present the commutator estimates used in the sequel. In Section \ref{sec:main}, we state and prove the main result of this article, whereas in Section \ref{models}, we study the fluid and solid models mentioned above. Section \ref{sec:triangular} is devoted to the construction of nontrivial examples for the triangular system, as well as the study of conditions allowing to extend solutions of one-dimensional problems to a multi-dimensional setting.


\section{Notation and preliminaries}\label{section:prelims}


We denote by $C^{k}(Q)$ the space of $k$-times continuously differentiable, $Q$-periodic functions and by $L^{p}(Q)$ the standard Lebesgue space of $Q$-periodic functions. Their norm is denoted by $\|\cdot\|_{L^{p}(Q)}$. In taking time into account, we consider the Bochner spaces $L^p(0,T;X)$, where $X$ is a Banach space, endowed with their standard norms.
We also denote by $C^{k}_c([0,T)]$ the space of $k$ times continuously differentiable functions, compactly supported on $[0,T)$, and naturally extended to define the space $C^k_c([0,T);C^k(Q))$.

Our main result on uniqueness concerns solutions that belong to an appropriate Besov space which we next define.

  \begin{definition}\label{def:besov}
    	Let $\alpha\in (0,1)$, $q\in [1,\infty)$ and $D\subset \R^M$ a bounded domain. Let $D_1\subset\re^M$ be open such that $\bar{D}\subset D_1$. The Besov space $B^{\alpha,\f}_q(D)$ is defined as the set of functions $g\in L^q(D)$ such that 
    	\begin{equation}
    	|g|_{B^{\alpha,\f}_q(D)}:=\sup\limits_{0\neq\xi\in\re^{M},D+\xi\subset D_1}\frac{\|g(\cdot+\xi)-g(\cdot)\|_{L^q(D)}}{\abs{\xi}^{\alpha}}<\f.
    	\end{equation} 
	$B^{\alpha,\f}_q(D)$ becomes a Banach space when equipped with the norm $\|\cdot\|_{B^{\alpha,\f}_q(D)} = \|\cdot\|_{L^q(D)} + |\cdot|_{B^{\alpha,\f}_q(D)}$. 
    \end{definition}
    Besov spaces enjoy the following property, see \cite{CET}: let $\zeta_\e$ be a sequence of mollifiers {\color{black} in space and time} and set $g_\e=g*\zeta_\e$. It holds that
        \begin{eqnarray}
    \|g_\e-g\|_{L^q(D)}&\leq &\e^{\al}|g|_{B^{\alpha,\f}_q(D)},\label{Est1}\\
    \|\nabla g_\e\|_{L^q(D)}&\leq &\e^{\al-1}|g|_{B^{\alpha,\f}_q(D)}.\label{Est2}
    \end{eqnarray}
    The estimates \eqref{Est1} and \eqref{Est2} result in the following lemma which is crucial in our analysis (see \cite{FGGW,FGJ} for a proof):
          \begin{lemma}[Commutator estimate \cite{CET,FGGW,FGJ}]\label{lemma_commutator}
    	Let $D\subset\R^M$ be a bounded domain. Let $D_1\subset\re^M$ be open such that $\bar{D}\subset D_1$. Suppose $w:D_1\to\R^m$ with $w\in B^{\alpha,\f}_{q}(D,\R^m)$ for $q\geq2$ and $\alpha\in(0,1)$. Let $\mathbb{B}\in C^2(K)$ where $K\subset\R^m$ be an open convex set containing the closure of the image of $w$. Let $\zeta_\e$ be a sequence of mollifiers with support in $\{\abs{x}<\e\}\subset \R^M$. Then 
    	\begin{equation}
    	\|\nabla_y(\mathbb{B}(w)_\e)-\nabla_y(\mathbb{B}(w_\e))\|_{L^{\frac{q}{2}}(D,\R^M)}\leq C\e^{2\alpha-1}\left(1+|w|_{B^{\alpha,\f}_q(D,\R^m)}^2\right)
    	\end{equation}
    	where $g_\e=g*\zeta_\e$ and $C=C(\|\mathbb{B}\|_{C^2(K)})$.
    \end{lemma}
{\color{black}We refer the reader to \cite{BGSTW,Deb,GMS} for similar commutator estimates in the context of Onsager's conjecture on the energy/entropy equality for various systems.}


In the sequel, we consider dissipative solutions to system \eqref{eqn:conlawgen} which we now define. We recall that $\mathcal{O}\subset \R^m$ is an open, convex set and we later impose the assumption that the functions $A$, $F_k$, and $H$ are continuous on $\mathcal{O}$, see \descref{H0}{(H0)}.
	
\begin{definition}\label{def:weak}
Assume that $H(U_0) \in L^1(\Dom)$.
\begin{itemize}
\item We say that $U:Q\times(0,T)\to\overline{\mathcal{O}}$ is a weak solution to \eqref{eqn:conlawgen} with initial data $U_0$ if
\begin{equation}\label{eqn:weak_formulation}
\int\limits_{\Dom}A(U_0)\cdot\Psi(\cdot,0)\,dx+\int\limits_{0}^{\tau}\int\limits_{\Dom}\left(A(U)\cdot\pa_t\Psi+F_k(U)\cdot \pa_{k}\Psi\right)\,dxdt = \int\limits_{\Dom}A(U(x,\tau))\cdot\Psi(\cdot,\tau)\,dx,
\end{equation}
for all $\Psi\in C^1(Q\times [0,T])$ and a.a. $\tau\in [0,T)$.
\item We say that $U:Q\times(0,T)\to\overline{\mathcal{O}}$ is a \textit{dissipative solution} of \eqref{eqn:conlawgen} with initial data $U_0$ if $U$ is a weak solution and satisfies the dissipation inequality
	\begin{equation}\label{ineq:entropy}
	\int\limits_Q H(U(x,\tau))\,dx \leq \int\limits_Q H(U(x,s))\,dx \leq \int\limits_Q H(U_0(x))\,dx
	\end{equation}
	for a.a. $0 < s < \tau < T$. 
	\end{itemize}
\end{definition} 

{\color{black}
\begin{remark}\label{remark0}
We note that as soon as $\eta = H\circ A^{-1}$ is $L^p$ coercive, $p>1$, the dissipation inequality says that
\[
A(U) \in L^\f(0,T,L^p(Q)).
\]
Moreover, combined with the equations, $L^p$ coercivity of $\eta$ also asserts that
\[
A(U) \in {\color{black} C_{weak}(0,T;L^{p}(\Dom))},
\]
meaning that, as $s_n \to s$,
\[
\int_\Dom (A(U(\cdot, s_n)) - A(U(\cdot,s)))\cdot \Phi \to 0,\mbox{ for all }\Phi\in L^{\frac{p}{p-1}}(Q).
\]
\end{remark}

\begin{remark}\label{remark1}
The above definition is consistent with the definition of admissible solution for the isentropic Euler system found in \cite{FGJ}. In addition, for reasonable growth conditions (see \descref{H2}{(H2)}), the definition of dissipative solution follows from the standard definition that $H(U)\in L^\f(0,T;L^1(\Dom))$ and
\begin{equation}\label{eqn:weak_formulation_alt}
\int\limits_{\Dom}A(U_0)\cdot\Psi(\cdot,0)\,dx+\int\limits_{0}^{T}\int\limits_{\Dom}\left(A(U)\cdot\pa_t\Psi+F_k(U)\cdot \pa_{k}\Psi\right)\,dxdt = 0,
\end{equation}
for all $\Psi\in C^1_c([0,T);C^1(Q))$ with the dissipation inequality
\begin{equation}\label{ineq:entropy_alt}
\int_0^T\int_Q\frac{d\theta}{dt} H(U)\,dxdt + \int_Q\theta(0)H(U_0)\,dx \geq 0,
\end{equation}
for all nonnegative functions $\theta\in C^1_c([0,T))$. Indeed, as we will assume in \descref{H2}{(H2)} (see \eqref{eq:growth_condition}), suppose that
\[
 |F_k(\xi)| + |A(\xi)| \lesssim 1+ H(\xi) .
\]
For $0\leq \tau < T$ fixed, let $(\theta_j)\subset C^1_c([0,T))$ be a bounded sequence, approximating the function
\[
\theta(t) = \left\{\begin{array}{cc}
1,&t\in[0,\tau)\\
(\tau - t)/\delta + 1,&t\in[\tau,\tau+\delta)\\
0,&t\in [\tau+\delta,T)
\end{array}\right.
\]
such that $(\theta_j)$ is nonincreasing and $\dot\theta_j(t)\rightarrow\dot\theta(t)$ for all $t\neq \tau,\tau+\delta$. For simplicity, we also assume that $\theta_j(0)=1$ for all $j$. Then, given $\Phi\in C^1(Q\times [0,T])$, test \eqref{eqn:weak_formulation_alt} with $\Psi = \theta_j \Phi$ to infer that
\begin{align*}
\int\limits_{\Dom}A(U_0)\cdot\Phi(\cdot,0)\,dx+\int\limits_{0}^{T}\int\limits_{\Dom}\theta_j\left(A(U)\cdot\pa_t\Phi+F_k(U)\cdot \pa_{k}\Phi\right)\,dxdt =  \int\limits_{0}^{T}\int\limits_{\Dom}|\dot\theta_j| \,A(U)\cdot\Phi,
\end{align*}
where $\dot\theta = d\theta/dt$. Next note that, since $\theta_j$ is bounded in $C^1$, the functions
\[
t\mapsto |\dot\theta_j| \int\limits_Q A(U)\cdot \Phi\mbox{ and }t\mapsto \theta_j\int\limits_Q A(U)\cdot\pa_t\Phi+F_k(U)\cdot \pa_{k}\Phi
\]
are both bounded (up to a constant) by $\|\Psi\|_{C^1}\int_Q 1+ H(U)(t) \in L^\infty((0,T))$. Hence, by dominated convergence, we may take the limit $j\to\infty$ to infer that
\begin{align*}
\int\limits_{\Dom}A(U_0)\cdot\Phi(\cdot,0) +\int\limits_{0}^{\tau + \delta}\int\limits_{\Dom}\left(A(U)\cdot\pa_t\Phi+F_k(U)\cdot \pa_{k}\Phi\right) & + \frac{1}{\delta}\int\limits_{\tau}^{\tau + \delta}(\tau - t)\int\limits_{\Dom}A(U)\cdot\pa_t\Phi+F_k(U)\cdot \pa_{k}\Phi \\
& =  \frac{1}{\delta}\int\limits_{\tau}^{\tau+\delta}\int\limits_{\Dom} A(U)\cdot\Phi.
\end{align*}
Again due to the fact that $A(U)$, $F_k(U)\in L^\f(0,T;L^1(\Dom))$, the functions
\[
t\mapsto (\tau - t)\int\limits_{\Dom}A(U)\cdot\pa_t\Phi+F_k(U)\cdot \pa_{k}\Phi \mbox{ and } t\mapsto \int\limits_{\Dom} A(U)\cdot\Phi
\]
are integrable in $(0,T)$ and by Lebesgue's differentiation theorem we may take the limit $\delta\to 0$ to deduce that for a.a. $0\leq \tau < T$,
\begin{align*}
\int\limits_{\Dom}A(U_0)\cdot\Phi(\cdot,0) +\int\limits_{0}^{\tau}\int\limits_{\Dom}\left(A(U)\cdot\pa_t\Phi+F_k(U)\cdot \pa_{k}\Phi\right) =  \int\limits_{\Dom} A(U)(\cdot,\tau)\cdot\Phi(\cdot,\tau).
\end{align*}
Thus, noting that the space $C^1(Q\times[0,T])$ is separable, we may choose a null set of times outside which the above inequality, that coincides with \eqref{eqn:weak_formulation}, holds. Similarly, we may infer the assumed dissipation inequality. In particular, for $0< s < \tau < T$, let $(\theta_j)\subset C^1_c([0,T))$ be a bounded sequence, approximating the function
\[
\theta(t) = \left\{\begin{array}{cc}
0,&t\in[0,s)\\
(t-s)/\delta, & t\in [s,s + \delta)\\
1, & t\in[s+\delta, \tau)\\
(\tau - t)/\delta + 1,&t\in[\tau,\tau+\delta)\\
0,&t\in [\tau+\delta,T)
\end{array}\right.
\]
such that $\dot\theta_j(t)\rightarrow\dot\theta(t)$ for all $t\neq s, s+\delta, \tau,\tau+\delta$. For simplicity, we also assume that $\theta_j(0)=0$ for all $j$. Testing the dissipation inequality \eqref{ineq:entropy_alt} with $\theta_j$ and passing to the limit in $j$ via dominated convergence, we find that
\[
\frac{1}{\delta}\int\limits_{s}^{s+\delta}\int\limits_Q \eta(U) -  \frac{1}{\delta}\int\limits_{\tau}^{\tau+\delta}\int\limits_Q \eta(U) \geq 0
\]
or equivalently \eqref{ineq:entropy} being understood that the argument also extends to $s=0$. 

Note also that dissipative solutions are less restrictive than entropic solutions, i.e. weak solutions satisfying \eqref{eq:entropy_ineq gen} when tested against functions in $C^1_c([0,T);C^1(Q))$. In particular, uniqueness within the class of dissipative solutions implies uniqueness within the class of entropic solutions.
\end{remark}
}

Our uniqueness result utilises the relative entropy method and below we provide some necessary terminology. For two vectors $\xi,\bar\xi\in\mathcal{O}$, we denote by $H(\xi|\bar\xi)$ the relative entropy defined by
\begin{equation}
H(\xi|\bar\xi):=H(\xi)-H(\bar\xi)-G(\bar\xi)\cdot(A(\xi) - A(\bar\xi)).
\end{equation}
Note that for $f_k = F_k \circ A^{-1}$, $\eta = H\circ A^{-1}$ and $q_k = Q_k\circ A^{-1}$, we find that
\[
\eta(A(\xi)|A(\bar\xi))=\eta(A(\xi))-\eta(A(\bar\xi))-D\eta(A(\bar\xi))\cdot(A(\xi) - A(\bar\xi)),
\] 
reducing to the standard relative entropy for $A(\xi) = \xi$. To present the relative entropy method, let us assume for simplicity that both weak and strong solutions lie within a compact of $\mathcal{O}$. Note that, for $z \in \mathcal{O}$,
\[
\left(D^2H(z) - G(z) D^2A(z) \right)_{ij} = \left(\pa_i A(z)\right)^T D^2\eta(A(z)) \left(\pa_j A(z)\right).  
\]
Since, $DA(z)$ is nonsingular in $\mathcal{O}$, its columns $\pa_i A(z)$ form a basis and given a vector $\zeta \in \re^m$ we find $\xi = (\xi_1,\cdots,\xi_m)^T$ such that
$\zeta = \xi_i \pa_i A(z)$,
where we have employed the Einstein summation convention. Hence, whenever $D^2H - G D^2A > 0$, we also find that
\begin{align*}
\zeta^T D^2\eta(A(z))\zeta 
& = \xi^T \left(D^2H(z) - G(z) D^2A(z) \right)\xi > 0.
\end{align*}
In particular, at least for $\xi$, $\bar \xi$ within a compact subset of $\mathcal{O}$, we infer that
\begin{equation}\label{eq:convex_etarel}
H(\xi|\bar\xi) \gtrsim \abs{A(\xi) - A(\bar\xi)}^2.
\end{equation}
The reader is referred to Lemma \ref{lemma:etarel_lower_bound} for a precise statement under weaker assumptions that are required for our purposes. 
Another quantity which plays a crucial role is the relative flux, defined for each $k=1,\cdots, d$ by
\begin{equation}\label{eqn:defn_Z_k}	
F_k(\xi|\bar\xi):=F_k(\xi)-F_k(\bar\xi)-D F_k(\bar\xi)DA(\bar\xi)^{-1}(A(\xi) - A(\bar\xi)),
\end{equation}
which can also be written as
\[
F_k(\xi|\bar\xi) = f_k(A(\xi)|A(\bar\xi)):=f_k(A(\xi))-f_k(A(\bar\xi))-D f_k(A(\bar\xi))(A(\xi) - A(\bar\xi)).
\]
Note that
\[
f_k(z|\bar z) = \int_0^1(1-\sigma) D^2f_k(\bar z + \sigma (z - \bar z))\,d\sigma (z - \bar z) \cdot (z - \bar z)
\]
and by \eqref{eq:convex_etarel} we find that at least for $\xi$, $\bar \xi$ within a compact subset of $\mathcal{O}$, 
\begin{equation}\label{eq:FkH}
|F_k(\xi|\bar\xi)| = |f_k(A(\xi)|A(\bar\xi)) | \lesssim |A(\xi) - A(\bar\xi)|^2 \lesssim H(\xi|\bar{\xi}).
\end{equation}
This estimate plays a crucial role in the application of the relative entropy method, see Lemma \ref{lemma:growthfk} for a proof of \eqref{eq:FkH} under the weaker assumptions employed here. Indeed, as it will become apparent from the proof of Theorem \ref{theorem1}, the relative entropy method leads to the following \emph{relative entropy inequality}:
\begin{equation*}
\int\limits_{\Dom}H(U|\bar{U})(x,\tau)\,dx\leq\int\limits_{\Dom}H(U|\bar{U})(x,0)\,dx-\int\limits_{0}^{\tau}\int\limits_{Q}\left[\pa_{k} G(\bar{U})\right]\cdot F_k(U|\bar{U})(x,t)\,dx dt,
\end{equation*}
where $U$ is an assumed dissipative solution and $\bar{U}$ a strong solution, i.e. $W^{1,\infty}(\overline{Q}\times[0,T])$, which lies in a compact $K\subset\mathcal{O}$. The idea is to use the convexity of $\eta = H\circ A^{-1}$, the quadratic nature of $F_k$ \eqref{eq:FkH}, and the regularity of $\bar U$ to estimate that, at least for $U$ within a compact of $\mathcal{O}$,
\begin{equation}\label{eq:relentineq0}
\int\limits_{\Dom} H(U|\bar{U})(x,\tau) \leq\int\limits_{\Dom}H(U|\bar{U})(x,0)+ C(\|\bar{U}\|_{W^{1,\infty}})\int\limits_{0}^{\tau}\int\limits_{Q} H(U|\bar{U})(x,t).
\end{equation}
The (weak-strong) uniqueness can be concluded, if $U(\cdot,0) = \bar{U}(\cdot,0)$, by Gr\"onwall's inequality. 

However, for general hyperbolic systems, we are unable to control that the weak solution remains within any compact of $\mathcal{O}$. In fact, weak solutions may even blow up in $L^{\f}$ at finite time for bounded initial data, see \cite{Baiti}. Thus, no $L^\f$ bounds can be assumed on dissipative solutions and appropriate growth, and coercivity, conditions need to be involved. Indeed, henceforth, we make the following assumptions which are partly motivated by \cite{tzavaras_cleopatra,GKS} and we refer the reader to Sections \ref{models}, \ref{sec:triangular} for relevant examples.

\begin{assumption} 
We assume the following on $G$, $H$ (resp. $\eta$), $F_k$ (resp. $f_k$) and $A$: 
\begin{description}
\descitem{(H0)}{H0} (regularity) $A,\,F_k,\,H \in C^2(\mathcal{O})$ for $k=1,\cdots, d$ and $A,\,F_k,\,H$ are continuous on $\overline{\mathcal{O}}$, where the sets $\mathcal{O}, A(\mathcal{O})\subset \R^m$ are assumed open and convex. Moreover, we assume that $DA(U)$ is nonsingular for $U\in\mathcal{O}$.
\descitem{(H1)}{H1} (coercivity) We assume that $H$ satisfies the coercivity condition
\begin{equation}\label{eq:coercivity_condition}
H(\xi) = \eta(A(\xi)) \gtrsim -1 + |A(\xi)|^{p}, \,\,p>1.
\end{equation}
\descitem{(H2)}{H2} (growth) For $F_k = f_k\circ A$, $H = \eta\circ A$, and $G = D\eta\circ A$ we assume the following:
\begin{description}
\descitem{(H2a)}{H2a} For some $l > 1$,
\begin{equation}\label{eq:growth_condition_1}
|H(\xi)| + |G(\xi)| \lesssim 1 + |\xi|^l.
\end{equation}
We note that we pose no restriction on the size of the exponent $l$.
\descitem{(H2b)}{H2b} For $F_k$ it holds that
\begin{equation}\label{eq:growth_condition}
 |F_k(\xi)| + |A(\xi)| = |f_k(A(\xi))| + |A(\xi)| \lesssim 1 + \eta(A(\xi)) = 1 + H(\xi) .
 \end{equation}
In particular, by \eqref{eq:growth_condition_1}, $F_k$ and $A$ also have polynomial growth.
\descitem{(H2c)}{H2c} Let $I \subset \{1,\ldots, d\}$ a set of indices such that the component $G_i$ of $G$ is nonlinear for $i\in I$. If $\left(F_k(\xi|\bar{\xi})\right)_i \equiv 0$ for all $i\in I$, we make no further assumptions. If for some $i\in I$, $\left(F_k(\xi|\bar{\xi})\right)_i \not\equiv 0$ then we strengthen \descref{H2b}{(H2b)} by assuming that for all such $i$
\begin{equation}\label{eq:extra_growth}
\abs{\left(F_k(\xi)\right)_i}^L \lesssim 1 + H(\xi),\quad\mbox{for some $L>1$}.
\end{equation}
\end{description}
\end{description}

\end{assumption}

\begin{remark}\label{rem:assumptions}
\quad
\begin{itemize}
\item Note that $\bar U$ is assumed to lie in a compact $K\subset\mathcal{O}$. Then, as in \cite{GKS}, we remark that for any $\bar{\xi}\in K$, the functions $\xi\mapsto H(\xi|\bar{\xi})$, and $\xi\mapsto F_k(\xi|\bar{\xi})$ are continuous on $\overline{\mathcal{O}}$. This follows from the fact that $H$ and $F_k$ are continuous on $\overline{\mathcal{O}}$ and that the maps $DF_k$, $DA^{-1}$, $G$ appear evaluated at $\bar{\xi}$ but not $\xi$.
\item Note that \eqref{eq:growth_condition} was already invoked in Remark \ref{remark1} and leads to the estimate (see Lemma \ref{lemma:growthfk})
\begin{equation}\label{eq:relflux_growth}
|F_k(\xi|\bar \xi)| =  f_k(A(\xi)|A(\bar \xi)) \lesssim \eta(A(\xi)|A(\bar \xi)) = H(\xi|\bar \xi).
\end{equation}
\end{itemize}
\end{remark}

Hence, by \eqref{eq:relflux_growth} we may reach inequality \eqref{eq:relentineq0} without the $L^\f$ assumption
and then conclude uniqueness provided $H(\xi|\bar \xi)$ vanishes only when $\xi = \bar \xi$ which is shown in Lemma \ref{lemma:etarel_lower_bound}. We note that estimate \eqref{eq:relentineq0}  is precisely where the regularity of the strong solution enters and it is this point that needs to be overcome, if the regularity of the strong solution is reduced. In particular, in the present article, we show that if $\bar{U}$ is merely in an appropriate Besov space, the uniqueness proof can be concluded under the condition:
\begin{equation}\label{ineq:general}\tag{OSC1}
\left[\pa_{k}G(\bar{U})\right]\cdot F_k(\xi|\bar{\xi}) + b(t) H(\xi|\bar \xi) \geq 0 \mbox{ in }\mathcal{D}^{\p}(\re^d)\mbox{ for all }(\xi,\bar{\xi})\in \overline{\mathcal{O}}\times\mathcal{O},
\end{equation}
which does not require any differentiability properties for $\bar{U}$. Indeed, \eqref{ineq:general} generalises the condition established in \cite{FGJ} for isentropic Euler, see \S \ref{sec:isnE}.

\begin{remark}
We note that \eqref{ineq:general} replaces the Lipschitz condition on $\bar{U}$ and thus also eliminates the need for assumption \descref{H2b}{(H2b)} which was invoked to prove the estimate $|F_k(\xi|\bar\xi)| \lesssim H(\xi|\bar\xi)$ of Lemma \ref{lemma:growthfk} below. Indeed, this estimate is only required in the relative entropy method to write \eqref{eq:relentineq0} for a Lipschitz solutions and, in our case, to also guarantee that any Lipschitz solution satisfies \eqref{ineq:general}. Moreover, assumption \descref{H2b}{(H2b)} is required to justify Definition \ref{def:weak}, see Remark \ref{remark1}. Hence, we prefer to include it in our list of assumptions.
\end{remark}

\begin{lemma}\label{lemma:etarel_lower_bound}
Suppose that \descref{H0}{(H0)}, \descref{H1}{(H1)} are satisfied and that $D^2\eta (z) > 0$ for all $z\in A(\mathcal{O})$. Let $K \subset \mathcal{O}$ compact. Then for all $\bar{\xi}\in K$ and $\xi\in\overline{\mathcal{O}}$ it holds that
\begin{equation}\label{eq:etarel_lower_bound}
 H(\xi|\bar \xi) = \eta(A(\xi)|A(\bar \xi)) \geq 0 \mbox{ and }  H(\xi|\bar \xi) = 0 \,\Leftrightarrow \, \xi = \bar\xi.
\end{equation}
\end{lemma}

\begin{proof}
We present an argument which is a modification of the proof of Lemma A.1 in \cite{GKS}. Let $\xi \in \mathcal{O}$ and let $\delta>0$ small enough such that
\[
K_{\delta} := \left\{z + w : z\in K,\,|w|\leq \delta\right\}\subset \mathcal{O}.
\]
We consider two cases: (a) $\xi \in K_\delta$ and (b) $\xi\in \mathcal{O}\setminus K_\delta$. Assume that $\xi\in K_\delta$. Denoting by ${\rm co}(B)$ the closed convex hull of a compact set $B$, we find that for any $s\in[0,1]$
\[
A(\bar\xi) + s(A(\xi) - A(\bar\xi)) \in {\rm co}\left(A(K_\delta)\right)\subset A(\mathcal{O})
\]
as $A(\mathcal{O})$ is itself convex. Then,
\begin{align}\label{eq:lemma_eta_rel_new1}
H(\xi|\bar\xi) & = \int_0^1(1-s) D^2\eta \left(A(\bar \xi)+s(A(\xi)-A(\bar \xi))\right)\,ds (A(\xi)-A(\bar \xi))\cdot(A(\xi)-A(\bar \xi))\nonumber \\
& \geq \frac12  \min_ {{\rm co}\left(A(K_\delta)\right) }\left\{|D^2\eta|\right\} |A(\xi)-A(\bar \xi)|^2 =: c_0 |A(\xi)-A(\bar \xi)|^2
\end{align}
where $c_0 > 0$ by the uniform convexity of $\eta$ on compact sets. Hence, the lemma follows in the case $\xi \in K_\delta$ as $DA$ is nonsingular on $\mathcal{O}$ and we may compute that
\begin{align*}
|\xi - \bar{\xi}| & \leq \int_0^1 DA^{-1}(A(\bar\xi) + s(A(\xi)-A(\bar \xi)))|\,ds |A(\xi) - A(\bar\xi)| \\
& \leq \max_ {{\rm co}\left(A(K_\delta)\right) }\left\{|(D A)^{-1}|\right\} |A(\xi) - A(\bar\xi) |.
\end{align*}

Next, assume that $\xi\in \mathcal{O}\setminus K_\delta$. Assume in addition that for some $s\in(0,1)$, $A(\bar \xi)+s(A(\xi)-A(\bar \xi))\notin {\rm co}\left(A(K_\delta)\right)$ as otherwise we may proceed as in case (a). Define
\[
s^*:=\inf\left\{s\in(0,1):A(\bar \xi)+s(A(\xi)-A(\bar \xi))\notin {\rm co}\left(A(K_\delta)\right) \right\}.
\]
Note that $s^*>0$ and that by the convexity of $A(\mathcal{O})$ there exists $\xi^*$ such that
\[
A(\xi^*) = A(\bar \xi)+s^*(A(\xi)-A(\bar \xi)).
\]
We also infer that $A(\xi^*)\in \partial \,{\rm co}\left(A(K_\delta)\right)$ and thus $A(\xi^*)$ cannot belong to the interior of $A(K_\delta)$. Then, since $DA$ is nonsingular in $\mathcal{O}$, $\xi^*$ cannot belong to the interior of $K_\delta$ and in particular $|\xi^* - \bar\xi | \geq \delta$. Moreover, the fact that $A(\xi),\,A(\xi^*)\in {\rm co}\left(A(K_\delta)\right)$, case (a), and the invertibility of $A$ imply that
\begin{align}\label{eq:lemma_eta_rel_new2}
H(\xi^*|\bar\xi) & \geq  c_0 |A(\xi^*)-A(\bar \xi))|^2 \geq \tilde\delta > 0,
\end{align} 
where $\tilde\delta$ does not depend on $\xi$. We now claim that $H(\xi|\bar\xi) \geq H(\xi^*|\bar\xi)$. Indeed, note that
\begin{align*}
H(\xi^*|\bar\xi) & = \int_0^1 \left(s^*\right)^2(1-s) D^2\eta(A(\bar\xi) + s s^* (A(\xi) - A(\bar\xi)))\,ds(A(\xi) - A(\bar\xi))\cdot (A(\xi) - A(\bar\xi)).
\end{align*}
Setting $\sigma = s s^*$ and bearing in mind that $D^2\eta$ is positive-definite on $A(\mathcal{O})$, we find that
\begin{align*}
H(\xi^*|\bar\xi) & = \int_0^{s^*} (s^*-\sigma) D^2\eta(A(\bar\xi) + \sigma (A(\xi) - A(\bar\xi)))\,d\sigma(A(\xi) - A(\bar\xi))\cdot (A(\xi) - A(\bar\xi))\\
& \leq \int_0^{s^*} (1-\sigma) D^2\eta(A(\bar\xi) + \sigma (A(\xi) - A(\bar\xi)))\,d\sigma(A(\xi) - A(\bar\xi))\cdot (A(\xi) - A(\bar\xi))\\
& \leq H(\xi|\bar\xi).
\end{align*}
Inequality \eqref{eq:lemma_eta_rel_new2} then says that for any $\xi\in\mathcal{O}\setminus K_\delta$,
\begin{equation}\label{eq:lemma_eta_rel_new3}
H(\xi|\bar\xi) \geq  \tilde\delta > 0.
\end{equation}
The continuity of $H(\cdot|\bar\xi)$ on $\overline{\mathcal{O}}$ now completes the proof.
\end{proof}

We end this section with a Lemma establishing \eqref{eq:relflux_growth} under \descref{H0}{(H0)}, \descref{H1}{(H1)}, and \descref{H2b}{(H2b)}.

\begin{lemma}\label{lemma:growthfk}
Suppose that \descref{H0}{(H0)}, \descref{H1}{(H1)}, and \descref{H2b}{(H2b)} are satisfied and let $K\subset\mathcal{O}$ compact. Then, for all $\bar{\xi}\in K$ and $\xi\in\overline{\mathcal{O}}$ it holds that
\[
|F_k(\xi|\bar \xi)| =  f_k(A(\xi)|A(\bar \xi)) \lesssim \eta(A(\xi)|A(\bar \xi)) = H(\xi|\bar \xi).
\]
\end{lemma}

\begin{proof}
Following the proof of Lemma \ref{lemma:etarel_lower_bound}, we consider two cases: (a) $\xi\in K_\delta$ and (b) $\xi\in \mathcal{O}\setminus K_\delta$ where we recall that 
\[
K_{\delta} = \left\{z + w : z\in K,\,|w|\leq \delta\right\}\subset \mathcal{O}.
\]
Suppose that $\xi\in K_\delta$. Then by \eqref{eq:lemma_eta_rel_new1} we may estimate that
\begin{align*}
|f_k(A(\xi)|A(\bar\xi)) | & = \left|\int_0^1(1-s) D^2f_k \left(A(\bar \xi)+s(A(\xi)-A(\bar \xi))\right)\,ds (A(\xi)-A(\bar \xi))\cdot(A(\xi)-A(\bar \xi))\right| \\
& \leq \sup_{{\rm co}(A(K_\delta))}|D^2f_k| \,|A(\xi) - A(\bar\xi)|^2 \lesssim \eta(A(\xi)|A(\bar\xi)).
\end{align*}
This completes the proof of case (a). Next, let $\xi\in \mathcal{O}\setminus K_\delta$ to find that
\begin{equation}\label{eq:lemma_flux_1}
F_k(\xi|\bar\xi) = f_k(A(\xi)|A(\bar\xi)) \lesssim |f_k(A(\xi))| + 1 + |A(\xi)| \lesssim 1 + \eta(A(\xi))
\end{equation}
where the suppressed constants in the first inequality only depend on the range of continuous functions on the compact set $K$, and the second inequality follows from \descref{H2b}{(H2b)}. We now utilise the coercivity condition \descref{H1}{(H1)}. In particular, Young's inequality says that
\begin{align*}
\eta(A(\xi)|A(\bar{\xi})) & \geq \eta(A(\xi)) - C - C(\delta) |D\eta(A(\bar \xi))|^{\frac{p}{p-1}} - \delta C |A(\xi)|^{p}\\
&  \geq \eta(A(\xi)) - C(\delta) - \delta C |A(\xi)|^{p}.
\end{align*}
However, the assumed coercivity condition states that
\[
|A(\xi)|^{p} \lesssim 1 + \eta(A(\xi)) = 1 + H(\xi),
\]
i.e. for $\delta>0$ small enough we find that
\begin{equation}\label{eq:lemma_flux_2}
H(\xi|\bar\xi) \gtrsim H(\xi) - 1.
\end{equation}
Hence, combining with \eqref{eq:lemma_flux_1}, we infer that
\begin{equation}\label{eq:lemma_flux_3}
F_k(\xi|\bar\xi) \lesssim 1 + H(\xi|\bar\xi).
\end{equation}
We remark that the coercivity condition and the resulting inequality \eqref{eq:lemma_flux_2}, are the ingredients replacing the condition $|A(\xi)|/\eta(\xi) \to 0$, as $|\xi| \to \infty$, found in \cite{GKS} and \cite[Lemma A.1]{tzavaras_cleopatra}. We are thus left to show that $H(\xi|\bar\xi) \gtrsim 1$ for $\xi \in \mathcal{O}\setminus K_\delta$. Indeed, as in the proof of Lemma \ref{lemma:etarel_lower_bound}, we may deduce \eqref{eq:lemma_eta_rel_new3}, i.e. that for some $\tilde\delta>0$,
\begin{equation*}
H(\xi|\bar\xi) \geq  \tilde\delta ,\mbox{ for any }\xi\in\mathcal{O}\setminus K_\delta.
\end{equation*}
The above inequality, \eqref{eq:lemma_flux_3} and the continuity of $H(\cdot|\bar\xi)$ and $F_k(\cdot|\bar\xi)$ on $\overline{\mathcal{O}}$ complete the proof.
\end{proof}


\section{Main result}\label{sec:main}

Our main result follows: 

    \begin{theorem}\label{theorem1}
    	Suppose that the system of conservation laws \eqref{eqn:conlawgen} is endowed with an entropy-entropy flux pair $(H,Q_k)$ satisfying \descref{H0}{(H0)}--\descref{H2}{(H2)} and \eqref{eq:convexity gen}. Let $U:\Dom\times [0,T]\to\overline{\mathcal{O}}$, $\bar{U}:\Dom\times [0,T]\to K\subset\mathcal{O}$, for $K$ compact, be dissipative solutions to \eqref{eqn:conlawgen} emanating from the initial data $U_0$ in the sense of Definition \ref{def:weak}, and suppose in addition the following:
   	\begin{enumerate}
   		\item $\bar{U}\in L^{\f}(\Dom\times(0,T))$ and for some $\alpha>\frac{1}{2}$
   		\begin{equation}\label{regularity_assumption}
   		\bar{U}\in B^{\alpha,\f}_q(\Dom\times(\de,T))\cap C(0,T;L^1(\Dom))\mbox{ for all }\de>0
   		\end{equation}
		where, as in \descref{H2c}{(H2c)}, for the indices $i\in I$ for which $(G)_i$ is nonlinear,
		\[
		q \geq \left\{\begin{array}{cl}
		 \max\left\{2p/(p - 1), 2{L}/(L-1)\right\}, & \mbox{ if } F_k(\xi|\bar{\xi})_i \not\equiv 0\mbox{ for some }i\in I\\
		 2p/(p - 1), &  \mbox{ if } F_k(\xi|\bar{\xi})_i \equiv 0\mbox{ for all }i\in I
		\end{array}\right.
		\]
		and $p>1$ as in \descref{H1}{(H1)}.
   		\item There exists $ b\in L^1((0,T))$ such that for $t\in(0,T)$, 
   		\begin{equation}\label{one_sided_bound_condition}\tag{OSC2}
   		\int\limits_{\Dom}\left(-\left[\pa_{k} \varphi(x)\right] G(\bar{U}(x,t))\cdot F_k(\xi|\bar{\xi})+b(t)\varphi(x) H(\xi|\bar\xi) \right)\,dx\geq0,
   		\end{equation}
   		for all $0\leq \varphi\in C^\f_c(\Dom)$ and $(\xi,\bar{\xi})\in\overline{\mathcal{O}}\times\mathcal{O}$ where $F_k$ is given by \eqref{eqn:defn_Z_k}.
   	\end{enumerate}
   	Then
   	\begin{equation*}
   	U(x,t)= \bar{U}(x,t)\mbox{ for a.e. }(x,t)\in\Dom\times(0,T).
   	\end{equation*}
   \end{theorem}
{\color{black} Note that \eqref{one_sided_bound_condition} and \eqref{ineq:general} are the same. We prefer to write \eqref{one_sided_bound_condition} to get an integral form and the negative sign in the first term comes from the definition of the distributional derivative.}

Next, we present a corollary to Theorem \ref{theorem1} when $A$ is linear and \eqref{eqn:conlawgen} reduces to \eqref{eqn:conlaw}. In this case, we show that the Besov regularity need only be assumed in the space variables. Note that we may now replace $H = \eta$, $F_k = f_k$, and $G = D\eta$.

 \begin{corollary}\label{theorem2}
Suppose that the system of conservation laws \eqref{eqn:conlaw} is endowed with an entropy-entropy flux pair $(\eta,q_k)$ satisfying \descref{H0}{(H0)}--\descref{H2}{(H2)} where $\eta$ is strictly convex on $\mathcal{O}$. Let $U:\Dom\times [0,T]\to\overline{\mathcal{O}}$, $\bar{U}:\Dom\times [0,T]\to K\subset\mathcal{O}$, for $K$ compact, be dissipative solutions to \eqref{eqn:conlaw} emanating from the initial data $U_0$ in the sense of Definition \ref{def:weak}, and suppose in addition the following:
   	\begin{enumerate}
   		\item $\bar{U}\in L^{\f}(\Dom\times(0,T))$ and for the same exponents $q$ and $\alpha$ as in Theorem \ref{theorem1}
		\begin{equation}
		\bar{U}\in L^{1}(\de,T;B^{\alpha,\f}_q(Q))\cap C(0,T;L^1(\Dom))\mbox{ for all }\de>0.
		\end{equation}
   		\item There exists $ b\in L^1((0,T))$ such that for $t\in(0,T)$, $\bar{U}$ satisfies \eqref{one_sided_bound_condition} for all $0\leq \varphi\in C^\f_c(\Dom)$ and {$(\xi,\bar{\xi})\in\overline{\mathcal{O}}\times\mathcal{O}$}.
   	\end{enumerate}
	Then
   	\begin{equation*}
   	U(x,t)= \bar{U}(x,t)\mbox{ for a.e. }(x,t)\in\Dom\times(0,T).
   	\end{equation*}
\end{corollary}

\begin{remark}
\label{rem:bounded}
We will see in Section \ref{models} that the isentropic Euler equations, as well as the system of convex elasticity and shallow water magnetohydrodynamics, indeed satisfy $F_k(\xi|\bar{\xi})_i\equiv 0$ for all $i$ such that $(G)_i$ is nonlinear. This is not true for the system of polyconvex elasticity which nevertheless satisfies \descref{H2c}{(H2c)}. Thus, we impose these assumptions as they appear naturally. However, it will become obvious from the proof that the conditions on $q$ in Theorem \ref{theorem1} and Corollary \ref{theorem2}, as well as condition \descref{H2c}{(H2c)}, are only required to show that the dissipative solution $U$ satisfies 
\[
F_k(U)_i,\,A(U) \in L^\f(0,T;L^{\frac{q}{q-2}}(Q)).
\]
In fact, in the case of Theorem \ref{theorem1}, the weaker condition
\[
F_k(U)_i \in L^{\frac{q}{q-2}}(Q\times (0,T))
\]
suffices, whereas for Corollary \ref{theorem2} one may assume that
\[
F_k(U)_i \in L^{\frac{r}{r-1}}(0,T;L^{\frac{q}{q-2}}(Q)),
\] 
together with $\bar{U}\in L^{r}(\de,T;B^{\alpha,\f}_q(Q))$. In particular, if $U\in L^\f(Q\times (0,T))$ lies within a compact of $\mathcal{O}$, we only need that $q > 2$ and no growth or coercivity conditions are required.  
\end{remark}

To aid the proof of Theorem \ref{theorem1} and in order to clarify the relative entropy method, we present the following Proposition:

\begin{proposition}\label{Prop:relative}
	Let $V\in C^1([0,T];C^1(Q))$ taking values in a compact $K\subset\mathcal{O}$ and $U$ a dissipative solution to \eqref{eqn:conlawgen}. Then, for $0\leq s<\tau \leq T$, the following form of the relative entropy inequality holds:
	\begin{equation}\label{ineq:relative_ent}
	\begin{array}{rl}
	\int\limits_{\Dom}H(U|V)(x,\tau)\,dx&\leq\int\limits_{\Dom}H(U|V)(x,s)\,dx\\
	&-\int\limits_{s}^{\tau}\int\limits_{\Dom}F_k(U)\cdot \pa_{k}G(V) - F_k(V)\cdot \pa_{k} G(V) + (A(U)-A(V))\cdot\pa_t G(V) dxdt.
	\end{array}
	\end{equation}
\end{proposition}

\begin{proof}
We note that using integration by parts, \eqref{eq:entropy-entropy_flux gen} and the $Q$-periodicity of $V$, we infer that
\begin{align}
\int\limits_{s}^{\tau}\int\limits_{\Dom}\pa_{k} G(V)\cdot F_k(V)\,dxdt
&=-\int\limits_{s}^{\tau}\int\limits_{\Dom} DF_k(V)^T G(V)\cdot \pa_{k}V\,dxdt\nonumber\\
&= -\int\limits_{s}^{\tau}\int\limits_{\Dom} D Q_k(V)\cdot\pa_{k}V\,dxdt \nonumber\\
&= -\int\limits_{s}^{\tau}\int\limits_{\Dom} \pa_{k} Q_k(V)\,dxdt=0.\label{eqn:prep2}
\end{align}

Next, test  \eqref{eqn:weak_formulation} with the function $G(V) \in C^1([0,T];C^1(Q))$ to obtain
\begin{equation}\label{Prop:step1}
\int\limits_{\Dom} A(U(x,\tau))\cdot G(V(x,\tau))  = \int\limits_{0}^{\tau}\int\limits_{\Dom} A(U)\cdot\pa_t G(V)+F_k(U)\cdot \pa_{k}G(V)
  + \int\limits_{\Dom}  A(U_0(x))\cdot G(V(x,0)) .
\end{equation}
Similarly, for $\tau=s$, we find that
\begin{equation}\label{Prop:step1.1}
\int\limits_{\Dom} A(U(x,s))\cdot G(V(x,s))  = \int\limits_{0}^{s}\int\limits_{\Dom} A(U)\cdot\pa_t G(V)+F_k(U)\cdot \pa_{k}G(V)
  + \int\limits_{\Dom}  A(U_0(x))\cdot G(V(x,0))
\end{equation}
and thus, subtracting \eqref{Prop:step1.1} from \eqref{Prop:step1} we infer that
\begin{equation}\label{Prop:step1.2}
\int\limits_{\Dom} A(U(x,\tau))\cdot G(V(x,\tau)) -  A(U(x,s))\cdot G(V(x,s)) = \int\limits_{s}^{\tau}\int\limits_{\Dom} A(U)\cdot\pa_t G(V)+F_k(U)\cdot \pa_{k}G(V).
\end{equation}
Moreover, since $V$ is smooth, a simple application of the fundamental theorem of calculus says that
\begin{equation}\label{Prop:step2}
		\begin{array}{rl}
		\int\limits_{\Dom}\left[A(V)\cdot G(V)-H(V)\right](x,\tau)\,dx=&\int\limits_{\Dom}\left[A(V)\cdot G(V)-H(V)\right](x,s)\,dx\\
		+&\int\limits_{s}^{\tau}\int\limits_{\Dom}A(V)\cdot\pa_t G(V)\,dxdt,
		\end{array}
		\end{equation}
where we have used \eqref{eq:entropy-entropy_flux gen} to write $\pa_t (G(V)\cdot A(V)) = A(V)\cdot \pa_t G(V) + \pa_tH(V)$. Lastly, recall the dissipation inequality \eqref{ineq:entropy}
\[
\int\limits_{\Dom} H(U(x,\tau))\,dx \leq \int\limits_{\Dom} H(U(x,s))\,dx{\mbox{ for }s<\tau},
\]
which combined with \eqref{Prop:step1.2} and \eqref{Prop:step2}, results in
		\begin{equation*}
		\begin{array}{rl}
			\int\limits_{\Dom}H(U|V)(x,\tau)\,dx&\leq\int\limits_{\Dom}H(U|V)(x,s)\,dx\\
			&-\int\limits_{s}^{\tau}\int\limits_{\Dom}F_k(U)\cdot \pa_{k}G(V) + (A(U)-A(V))\cdot\pa_t G(V)dxdt.
			\end{array}
		\end{equation*}
Together with \eqref{eqn:prep2}, the above inequality concludes the proof.
\end{proof}

We may now proceed to the proof of our main result.

\begin{proof}[Proof of Theorem \ref{theorem1}:]

		Let $U$, $\bar{U}$ as in the statement. We wish to apply Proposition \ref{Prop:relative} to $\bar{U}$ which, however, lacks regularity. We instead consider a sequence of mollifiers (in time and space) $\zeta_\eps$ to find that
		\begin{equation*}
		\pa_t A(\bar{U})_\e+\pa_{k}F_k(\bar{U})_\e=0
		\end{equation*}
		where $A(\bar{U})_\e=A(\bar{U})*\zeta_\e$ and $F_k(\bar{U})_\e=F_k(\bar{U})*\zeta_\e$. Thus, for $\bar{U}_\e=\bar{U}*\zeta_\e$,
		\begin{equation}\label{eqn_mollified}
		\pa_t A(\bar{U}_\e)+\pa_{k}F_k(\bar{U}_\e)=\mathcal{R}^\e
		\end{equation}
		where 
		\begin{equation}\label{eq:repsilon}
		\mathcal{R}^\e=\pa_t(A(\bar{U}_\e) - A(\bar{U})_\e) + \pa_{k}(F_k(\bar{U}_\e)-F_k(\bar{U})_\e) 
		\end{equation}
		and we may apply Proposition \ref{Prop:relative} with $V=\bar{U}_\e$ to infer that for $0<s<\tau<T$,
		\begin{equation}\label{eqn:intm_step1}
		\begin{array}{rl}
		\int\limits_{\Dom}H(U|\bar{U}_\e)(x,\tau)\,dx&\leq\int\limits_{\Dom}H(U|\bar{U}_\e)(x,s)\,dx\\
		&-\int\limits_{s}^{\tau}\int\limits_{\Dom}F_k(U)\cdot \pa_{k}G(\bar{U}_\e) - F_k(\bar{U}_\e)\cdot \pa_{k}G(\bar{U}_\e) + (A(U) - A(\bar{U}_\e))\cdot\pa_t G(\bar{U}_\e)dxdt.
		\end{array}
		\end{equation}
		Next note that \eqref{eq:symmetricgen} and \eqref{eqn_mollified} dictate that
		\begin{align*}
		(A(U) - A(\bar{U}_\e))\cdot\pa_t G(\bar{U}_\e) & = (A(U) - A(\bar{U}_\e))\cdot  DG(\bar{U}_\e)\pa_t \bar{U}_\e \\
		& = -(A(U) - A(\bar{U}_\e))\cdot DG(\bar{U}_\e) DA(\bar{U}_\e)^{-1}DF_k(\bar{U}_\e)\pa_k \bar{U}_\e \\
		&\quad + (A(U) - A(\bar{U}_\e))\cdot DG(\bar{U}_\e) DA(\bar{U}_\e)^{-1}\mathcal{R}^\e\\
		& = -(A(U) - A(\bar{U}_\e))\cdot DA(\bar{U}_\e)^{-T} DG(\bar{U}_\e)^{T}DF_k(\bar{U}_\e)\pa_k \bar{U}_\e + \mathcal{S}^\e \\
		& = -(A(U) - A(\bar{U}_\e))\cdot DA(\bar{U}_\e)^{-T} DF_k(\bar{U}_\e)^{T}DG(\bar{U}_\e)\pa_k \bar{U}_\e + \mathcal{S}^\e \\
		& = -DF_k(\bar{U}_\e) DA(\bar{U}_\e)^{-1}(A(U) - A(\bar{U}_\e))\cdot \pa_k G(\bar{U}_\e) + \mathcal{S}^\e,
		\end{align*}
		where we set
		\begin{equation}\label{eq:sepsilon}
		\mathcal{S}^\e := (A(U) - A(\bar{U}_\e))\cdot DG(\bar{U}_\e) DA(\bar{U}_\e)^{-1}\mathcal{R}^\e.
		\end{equation}
		Hence, \eqref{eqn:intm_step1} becomes
		\begin{equation}\label{eqn:interm3}
		\begin{array}{rl}
		\int\limits_{\Dom}H(U|\bar{U}_\e)(x,\tau)\,dx&\leq\int\limits_{\Dom}H(U|\bar{U}_\e)(x,s)\,dx - \int\limits_{s}^{\tau}\int\limits_{\Dom}\left[\pa_{k} G(\bar{U}_\e)\right]\cdot F_k(U|\bar{U}_\e)\,dx 
		- \int\limits_{s}^{\tau}\int\limits_{\Dom} \mathcal{S}^\e,
		\end{array}
		\end{equation}	
where we used the definition of $F_k(\cdot|\cdot)$ as
\[
F_k(U|V) = F_k(U) - F_k(V) - DF_k(V) DA(V)^{-1}(A(U) - A(V)).
\]	
Using $\varphi(y,s)=\zeta_\e(x-y, t-s)$ in \eqref{one_sided_bound_condition} and integrating in $s$, we get that
	\begin{equation}\label{mollified_Z}
	\pa_{k}G(\bar{U})_\e \cdot F_k(\xi|\bar{\xi})+b_\e(t) H(\xi|\bar\xi) \geq0\mbox{ for }t>0\mbox{ and }x\in\Dom,
	\end{equation}
	 {\color{black}where the sign reversal in \eqref{mollified_Z} is due to the derivative being considered in the $y$ variable}. Hence letting $\xi=U$ and $\bar{\xi}=\bar{U}_\e$ in \eqref{mollified_Z} we infer that
	\begin{equation}\label{mollified_Z1}
	\pa_{k}G(\bar{U}_\e)\cdot F_k(U|\bar{U}_\e)+b_\e(t) H(U|\bar{U}_\e) \geq \mathcal{T}^\e,
	\end{equation}
	where
	\begin{equation}\label{def:R2}
	\mathcal{T}^\e := \left(\pa_k G(\bar{U}_\e) - \pa_k G(\bar U)_\e\right)\cdot F_k(U|\bar{U}_\e).
	\end{equation}
	Then, through \eqref{mollified_Z1}, \eqref{eqn:interm3} now reads
	\begin{equation}\label{eqn:interm4}
	\begin{array}{rl}
	\int\limits_{\Dom}H(U|\bar{U}_\e)(x,\tau)\,dx&\leq\int\limits_{\Dom}H(U|\bar{U}_\e)(x,s)\,dx + \int\limits_{s}^{\tau}\int\limits_{\Dom}b_\e(t) H(U|\bar{U}_\e) \,dx - \int\limits_{s}^{\tau}\int\limits_{\Dom}\mathcal{S}^\e+ \mathcal{T}^\e\,dxdt.
	\end{array}
	\end{equation}	
	By virtue of Lemma \ref{lemma_commutator}, the assumptions on $\bar{U}$ and $U$ we may pass to the limit $\e\rr0$ to get
	\begin{equation}\label{eqn:interm4.1}
	\int\limits_{\Dom}H(U|\bar{U})(x,\tau)\,dx\leq \int\limits_{\Dom}H(U|\bar{U})(x,s)\,dx + \int\limits_{s}^{\tau}\int\limits_{\Dom}b(t) H(U|\bar{U}) \,dxdt.
	\end{equation}
	Before we proceed, let us justify in detail \eqref{eqn:interm4.1} and point to the use of the assumptions stated in the theorem. First, note that $\bar U_\e \to \bar U$ for a.e. $(x,t)$ where $\bar U_\e$, $\bar U$ take values in a compact subset of $\mathcal{O}$. Moreover, due to the dissipation inequality \eqref{ineq:entropy}
\begin{equation}\label{eq:proof_H}
H(U) \in L^\f(0,T;L^1(\Dom))
\end{equation}
and by the coercivity condition \descref{H1}{(H1)} in \eqref{eq:coercivity_condition}, see also Remark \ref{remark0},
\begin{equation}\label{eq:proof_A}
A(U) \in L^\f(0,T;L^{p}(\Dom)),\mbox{ where }p>1.
\end{equation}
In particular, $H(U|\bar U_\e)(\cdot,t)\in L^\f(0,T;L^1(Q))$ and, by the smoothness of $H$, $G$, $A$, the dominated convergence theorem says that
\[
\int\limits_\Dom H(U|\bar U_\e)(\cdot, t) \to \int\limits_\Dom H(U|\bar U)(\cdot, t)\mbox{ for a.e. }t\in (0,T).
\]
Note that $b_\e\to b$ in $L^1((0,T))$ and invoking Vitali's convergence theorem, we find that
\[
\int\limits_s^\tau\int\limits_\Dom b_\e(t) H(U|\bar{U}_\e) \to \int_s^\tau\int_\Dom b(t) H(U|\bar U).
\]
Next, to establish \eqref{eqn:interm4.1}, we show that $\mathcal{S}^\e$, $\mathcal{T}^\e\to0$. To estimate $\mathcal{S}^\e$, recalling that $\bar{U}_\e$ lies in a compact, the continuity of $DG$, $A$ and $DA^{-1}$ ensure that
\[
\|\mathcal{S}^\e\|_{L^1(Q\times(s,\tau))} \leq C \int\limits_s^\tau\int\limits_\Dom |A(U) - A(\bar U_\e)| |\pa_t(A(\bar{U}_\e) - A(\bar U)_\e) + \partial_k (F_k(\bar U_\e) - F_k(\bar U)_\e)|.
\]
Now Lemma \ref{lemma_commutator} implies that
\begin{equation}\label{limit_pass:R1}
\|\mathcal{S}^\e\|_{L^1(Q\times(s,\tau))} \leq C \|A(U) - A(\bar U_\e)\|_{L^{\frac{q}{q-2}}(\Dom\times (s,\tau))} \e^{2\alpha-1}\left(1 + |\bar{U}|^2_{B^{\alpha,\f}_{q}(\Dom\times (s,\tau))}\right) \to 0,\mbox{ as }\e\to 0
\end{equation}
{\color{black} since, for $p \geq q/(q-2)$, $A(U) \in L^\f(0,T;L^{\frac{q}{q-2}}(\Dom))$}. We are thus left to estimate $\mathcal{T}^\e$. {In particular, by H\"older's inequality and Lemma \ref{lemma_commutator}
\begin{align}
\|\mathcal{T}^\e\|_{L^1(\Dom\times (s,\tau))} & \leq \sum_i \|\pa_k G_i(\bar{U}_\e) - \pa_k G_i(\bar U)_\e\|_{L^{\frac{q}{2}}(\Dom\times (s,\tau))}\|\left(F_k\right)_i(U|\bar{U}_\e)\|_{L^{\frac{q}{q-2}}(\Dom\times (s,\tau))}\nonumber \\
& \leq C \e^{2\alpha-1}\left(1 + |\bar{U}|^2_{B^{\alpha,\f}_{q}(\Dom\times (s,\tau))}\right)\, \sum_i\|\left(F_k\right)_i(U|\bar{U}_\e)\|_{L^{\frac{q}{q-2}}(\Dom\times (s,\tau))}\label{eq:estimate_R2}
\end{align}
where the index $i$ ranges among the components such that $G_i$ is nonlinear. Indeed, note that whenever $G_i$ is linear the commutator vanishes. However, recalling that $\bar{U}_\e$ lies in a compact subset of $\mathcal{O}$, we find that
\[
\left|\left(F_k\right)_i(U|\bar{U}_\e) \right| \lesssim |(F_k)_i(U)| + 1 + |A(U)|
\]
by the continuity of $DF_k$ and $DA^{-1}$. But we have already argued that $A(U) \in L^\f(0,T;L^{\frac{q}{q-2}}(\Dom))$ and from \descref{H2c}{(H2c)} it follows that
\[
(F_k)_i(U) \in L^\f(0,T;L^{L}(\Dom)).
\]
By the assumptions of Theorem \ref{theorem1}, note that $q/(q-2) \leq L$, so that
\[
(F_k)_i(U) \in L^\f(0,T;L^{\frac{q}{q-2}}(\Dom)).
\] 
}
Thus, the right-hand side of \eqref{eq:estimate_R2} converges to $0$ as $\e\to0$ and \eqref{eqn:interm4.1} is proved. Of course, if $(F_k)_i(\xi|\bar\xi) = 0$ the argument is simpler. Next, by the dissipation inequality \eqref{ineq:entropy} we have that
\begin{align}
	\int\limits_{\Dom}H(U|\bar{U})(x,\tau)\,dx & \leq  \int\limits_\Dom H(U_0(x)) - H(\bar U(x,s)) - G(\bar U(x,s))\cdot(A(U(x,s))-A(\bar U(x,s)))\,dx\nonumber \\
	&\quad + \int\limits_{s}^{\tau}\int\limits_{\Dom}b(t) H(U|\bar{U}) \,dxdt\label{eqn:interm4.2}
\end{align}
and we now wish to pass to the limit $s\to 0$. By \eqref{eq:proof_H}, \eqref{eq:proof_A} and the fact that $b\in L^1((0,T))$ we infer that, as $s\to 0$,
\begin{equation}\label{eq:sto0_1}
\int\limits_{s}^{\tau}\int\limits_{\Dom}b(t) H(U|\bar{U}) \to \int\limits_{0}^{\tau}\int\limits_{\Dom}b(t) H(U|\bar{U}).
\end{equation}
Moreover, $\bar U \in C(0,T; L^1(\Dom))$, i.e. since $\bar U(x,0) = U_0(x)$ and $\bar U\in L^\f(Q\times (0,T))$,
\[
\lim_{s\to 0}\int\limits_\Dom|\bar U(x,s) -  U_0(x)|^r = 0 \mbox{ for all }r < \f.
\]
Similarly, by the assumed polynomial growth on $H$, $G$ and $A$ in \descref{H2}{(H2)} and dominated convergence, we find that
\begin{align}
0 & = \lim_{s\to 0}\int\limits_\Dom |H(\bar U(x,s)) - H(U_0(x))|\nonumber \\
&= \lim_{s\to 0}\int\limits_\Dom |G(\bar U(x,s)) - G(U_0(x))|\nonumber \\
&= \lim_{s\to 0}\int\limits_\Dom |G(\bar U(x,s))A(\bar U(x,s)) - G(U_0(x)) A(U_0(x))|,\label{eq:sto0_2}
\end{align}
where the above convergences also hold in $L^r(\Dom)$, $r<\f$, since $\bar U \in L^\f(\Dom\times (0,T))$.
Next, {\color{black} note that by the assumed coercivity \descref{H1}{(H1)} of $H$ in \eqref{eq:coercivity_condition} and Remark \ref{remark0}, $A(U) \in C_{\rm weak}(0,T; L^{p}(\Dom))$}, where $p > 1$, i.e. 
\[
\lim_{s\to 0}\int\limits_\Dom  \varphi(x) \cdot (A(U(x,s)) - A(U_0(x))) \,dx \to 0,\mbox{ for all }\varphi\in L^{\frac{p}{p - 1}}.
\]
Together with the strong convergence $G(\bar U(\cdot,s)) \to G(U_0)$ in $L^r$, $r<\f$, we find that
\begin{equation}\label{eq:sto0_3}
\lim_{s\to 0}\int\limits_\Dom  G(\bar U(x,s)) \cdot A(U(x,s)) \,dx \to \lim_{s\to 0}\int\limits_\Dom  G( U_0(x)) \cdot A(U_0(x)) \,dx.
\end{equation}
Through \eqref{eq:sto0_1}--\eqref{eq:sto0_3}, \eqref{eqn:interm4.2} now reads
\begin{equation*}
\int\limits_{\Dom}H(U|\bar{U})(x,\tau)\,dx\leq \int\limits_{0}^{\tau}\int\limits_{\Dom}b(t) H(U|\bar U) \,dxdt
\end{equation*}
and Gr\"onwall's inequality says that
\[
\int_{\Dom} H(U|\bar{U})(x,\tau) \leq 0.
\]
Lemma \ref{lemma:etarel_lower_bound} concludes the proof.
\end{proof}

We next present a sketch of the proof of Corollary \ref{theorem2} which removes the Besov regularity in time when $A$ is linear.

{\color{black}\begin{proof}[Proof of Corollary \ref{theorem2}]
We may proceed exactly as in the proof of Theorem \ref{theorem1} to reach \eqref{eqn:interm4} and we need to justify the passage to \eqref{eqn:interm4.1}. Note that we need only justify that the error terms $\mathcal{S}_\e$, $\mathcal{T}_\e$ vanish in the limit, as all other terms do not involve the Besov regularity. To estimate $\mathcal{S}^\e$, note that $\pa_tA(\bar U)_\e=\pa_tA(\bar{U}_\e)$ for $A$ linear and thus
\[
\|\mathcal{S}^\e\|_{L^1(Q\times(s,\tau))} \leq C \int\limits_s^\tau\int\limits_\Dom |U - \bar U_\e| |\partial_k (f_k(\bar U_\e) - f_k(\bar U)_\e)|,
\]
replaces the estimate above \eqref{limit_pass:R1}. Then, by H\"older's inequality we find that
\begin{equation}\label{limit_pass:R1.1}
\|\mathcal{S}^\e\|_{L^1(Q\times(s,\tau))}  \leq C  \e^{2\alpha-1} \int_s^\tau \|U(\cdot,t) - \bar U_\e(\cdot,t)\|_{L^{\frac{q}{q-2}}(\Dom)}\left(1 + |\bar{U}(\cdot,t)|^2_{B^{\alpha,\f}_{q}(\Dom)}\right) \to 0,\,\,\e\to0
\end{equation}
since, for $p \geq q/(q-2)$, $U \in L^\f(0,T;L^{\frac{q}{q-2}}(\Dom))$ and $\bar{U} \in L^1(0,T; B^{\alpha,\f}_{q}(\Dom))$. For $\mathcal{T}^\e$, by H\"older's inequality, Lemma \ref{lemma_commutator} and for $G = D\eta$, we estimate
\begin{align}
\|\mathcal{T}^\e\|_{L^1(\Dom\times (s,\tau))} & \leq \sum_i  \int_s^\tau  \|\pa_k D\eta_i(\bar{U}_\e(\cdot,t)) - \pa_k D\eta_i(\bar U(\cdot,t))_\e\|_{L^{\frac{q}{2}}(\Dom)}\|\left(f_k\right)_i(U(\cdot,t)|\bar{U}_\e(\cdot,t))\|_{L^{\frac{q}{q-2}}(\Dom)}\nonumber \\
& \leq C \e^{2\alpha-1}\int_s^\tau \left(1 + |\bar{U}(\cdot,t)|^2_{B^{\alpha,\f}_{q}(\Dom)}\right)\, \sum_i\|\left(f_k\right)_i(U(\cdot,t)|\bar{U}_\e(\cdot,t))\|_{L^{\frac{q}{q-2}}(\Dom)}\nonumber\\
& \leq C  \e^{2\alpha-1} \left(\int_s^\tau \left(1 + |\bar{U}(\cdot,t)|^2_{B^{\alpha,\f}_{q}(\Dom)}\right) \right)  \sum_i \sup_t \|\left(f_k\right)_i(U(\cdot,t)|\bar{U}_\e(\cdot,t))\|_{L^{\frac{q}{q-2}}(\Dom)}.\nonumber
\label{eq:estimate_R2.1}
\end{align}
But $\bar U \in L^1(0,T; B^{\alpha,\f}_{q}(\Dom))$ and
\[
\left|\left(f_k\right)_i(U|\bar{U}_\e) \right| \lesssim |(f_k)_i(U)| + 1 + |U|
\]
where $U\in L^\f(0,T;L^{p}(\Dom))$ and, by \descref{H2c}{(H2c)}, also
\[
(f_k)_i(U) \in L^\f(0,T;L^{L}(\Dom)).
\] 
Since $q/(q-2) \leq \min\{p,L\}$, we deduce that $\mathcal{T}^\e \to 0$ in $L^1$ and the remaining proof proceeds exactly as in Theorem \ref{theorem1}.
\end{proof}
}


%
%
%
%
%
%


\section{Applications}\label{models}

In this section, we present some systems of conservation laws that fit into the general setting \eqref{eqn:conlawgen}. We show that they fulfil assumptions \descref{H0}{(H0)}--\descref{H2}{(H2)}, and express the one-sided condition \eqref{ineq:general} for the given systems. In particular, we recover the result in \cite{FGJ} for the isentropic Euler system, and provide new examples for the systems in elasticity and shallow water magnetohydrodynamics.


\subsection{Isentropic Euler system}\label{sec:isnE}
As a first application of our main result, we consider the isentropic Euler system taking the following form:
\begin{equation}\label{isentropic_euler}
\begin{array}{rll}
\pa_t \rho+\mbox{div}(\rho{v})&=&0,\\
\pa_t (\rho {v})+\mbox{div}(\rho{v}\otimes {v})+\nabla p(\rho)&=&0,
\end{array}\mbox{ for }(x,t)\in\Dom\times\re_+,
\end{equation}
where $p(\rho) = \rho^\gamma$, for some $\gamma>1$. Note that the above equations fit the general framework \eqref{eqn:conlawgen} with
\begin{equation}\label{def:isen:A-F}
U = \left(\begin{array}{c} \rho \\ v \end{array}\right),\quad A(U) = \left(\begin{array}{c} \rho \\ \rho v \end{array}\right),\quad F_k(U) = \left(\begin{array}{c} \rho v_k \\ \rho v v_k + p(\rho) e_k  \end{array}\right).
\end{equation}
The mass density $\rho$ is required to be positive and thus
\[
\mathcal{O} = \left\{(\rho,v)\in \re\times \re^d\,:\, \rho > 0\right\}.
\]
For system \eqref{isentropic_euler}, the functions $G$ and $H$ are given by 
\begin{equation}
G(U)=\left(P^{\p}(\rho)-\frac{1}{2}\abs{v}^2,v^{T}\right)\mbox{ and }H(U)=\frac{1}{2}\rho\abs{v}^2+P(\rho),
\end{equation}
where 
\[
P(\rho)=\rho\int\limits_{1}^{\rho}\frac{p(r)}{r^2}\,dr.
\]
Denoting by $\mathbb{I}_d$ the identity $d\times d$ matrix and by $e_k$ the $k$-th vector in the standard basis of $\re^d$, an elementary calculation gives that
\begin{align}
DA(U)&=\begin{pmatrix}
1&{0}^T\\
v&\rho\mathbb{I}_d
\end{pmatrix},\,
DA(U)^{-1}=\frac{1}{\rho}\begin{pmatrix}
\rho&{0}^T\\
-v&\mathbb{I}_d
\end{pmatrix},\label{derivative:A:isen}\\
DF_k(U)&=\begin{pmatrix}
v_k&\rho e_k^T\\
vv_k+p^{\p}(\rho)e_k&\rho v_k\mathbb{I}_d+\rho v\otimes e_k
\end{pmatrix}.\label{derivative:Fk:isen}
\end{align}
Moreover, we observe that
\begin{equation*}
G(U)D^2A(U)=\left(P^{\p}(\rho)-\frac{1}{2}\abs{v}^2\right)\begin{pmatrix}
0&{0}^T\\
0&\textbf{0}_d
\end{pmatrix}+v_k
\begin{pmatrix}
0&e_k^{T}\\
e_k&\textbf{0}_d
\end{pmatrix}=\begin{pmatrix}
0& v^{T}\\
v&\textbf{0}_d
\end{pmatrix},
\end{equation*}
where $\textbf{0}_d$ denotes the zero $d\times d$ matrix. Then, by the convexity of $P$ we find that
\begin{equation*}
D^2H-G(U)D^2A(U)=\begin{pmatrix}
P^{\p\p}(\rho)&  v^T\\
v&\rho\mathbb{I}_d
\end{pmatrix}-\begin{pmatrix}
0& v^{T}\\
v&\textbf{0}_d
\end{pmatrix}=\begin{pmatrix}
P^{\p\p}(\rho)&  0^T\\
0&\rho\mathbb{I}_d
\end{pmatrix}>0
\end{equation*}
in $\mathcal{O}$ and \eqref{eq:convexity gen} is satisfied. Next, set $p = 2\g/(\ga+1)$ and note that $\g > 2\g/(\g + 1)$ for $\g >1$. Then, by Young's inequality and the fact that $|z|^q \leq 1 + |z|^p$ for $p>q$, we may estimate that
\begin{align*}
|A(U)|^p & \lesssim \rho^p + (\rho |v|)^p \\
& \lesssim 1+ \rho^{\g}+\rho^{2\ga/(\ga+1)}\abs{v}^{2\g/(\g+1)}\\
& \lesssim 1 + \rho^{\g}+\rho\abs{v}^2\lesssim 1+ H(U).
\end{align*}
%
By a similar argument we have
\begin{align}
\abs{F_k(U)} & \leq \rho\abs{v_k}+\rho\abs{v_kv}+p(\rho)\nonumber\\
& \lesssim \rho+\rho\abs{v}^2+\rho^{\g}\nonumber\\
&\lesssim 1+\rho^{\g}+\rho\abs{v}^2\lesssim 1+H(U).
\end{align}
Therefore, assumptions \descref{H0}{(H0)}, \descref{H1}{(H1)}, \descref{H2a}{(H2a)}, and \descref{H2b}{(H2b)} are satisfied. Moreover, from \eqref{derivative:A:isen} and \eqref{derivative:Fk:isen} we have
\begin{align*}
DF_k(U)DA(U)^{-1}&=\frac{1}{\rho}\begin{pmatrix}
v_k&\rho e_k^T\\
vv_k+p^{\p}(\rho)e_k&\rho v_k\mathbb{I}_d+\rho v\otimes e_k
\end{pmatrix}\begin{pmatrix}
\rho&{0}^T\\
-v&\mathbb{I}_d
\end{pmatrix}\\&
=\begin{pmatrix}
0& e_k^T\\
-vv_k+p^{\p}(\rho)e_k&v_k\mathbb{I}_d+ v\otimes e_k
\end{pmatrix}.
\end{align*}
Recalling the definition of $F_k(\xi|\bar{\xi})$ as in \eqref{eqn:defn_Z_k}, we observe that
\begin{align*}
F_k({U}|\bar{U})&=F_k({U})-F_k(\bar{U})-DF_k(\bar{U})DA(\bar{U})^{-1}(A({U})-A(\bar{U}))\\
&=\left(\begin{array}{c} {\rho} {v}_k-\bar{\rho} \bar{v}_k \\ {\rho} {v} {v}_k + p({\rho}) e_k -\bar{\rho} \bar{v} \bar{v}_k - p(\bar{\rho}) e_k \end{array}\right)
-\begin{pmatrix}
0& e_k^T\\
-\bar{v}\bar{v}_k+p^{\p}(\bar{\rho})e_k&\bar{v}_k\mathbb{I}_d+ \bar{v}\otimes e_k
\end{pmatrix}\begin{pmatrix}
{\rho}-\bar{\rho}\\
{\rho}{v}-\bar{\rho}\bar{v}
\end{pmatrix}\\
&=\begin{pmatrix}
0\\
\bar{\rho} ({v}-\bar{v})( {v}_k-\bar{v}_k) 
\end{pmatrix}+
\begin{pmatrix}
0\\
(p({\rho}) -p(\bar{\rho})-({\rho} -\bar{\rho})p^{\p}(\bar{\rho}))e_k
\end{pmatrix}.
\end{align*}
Note that $(F_k)_1(U|\bar{U}) = 0$ where $(G)_1$ is the nonlinear component of $G$ and \descref{H2c}{(H2c)} is also satisfied. We note that the system can be expressed in alternative variables and we refer the reader to \cite{GKS} for an analysis as above.
Hence, the required Besov regularity is in $B^{\alpha,\f}_q$ where
\[
\frac{q}{2} \geq \frac{p}{p - 1} = p^\prime = \left(\frac{2\gamma}{\gamma + 1}\right)^\prime = \frac{2\gamma}{\gamma -1}
\]
which agrees with \cite{FGJ}. We also compute that
\begin{align*}
H(U|\bar U)&=H(U)-H(\bar{U})-G(\bar{U})(A({U})-A(\bar{U}))\\
&=\frac{1}{2}{\rho}\abs{{v}}^2+P({\rho})-\frac{1}{2}\bar{\rho}\abs{\bar{v}}^2-P(\bar{\rho})-\left(P^{\p}(\bar{\rho})-\frac{1}{2}\abs{\bar{v}}^2,\bar{v}^{T}\right)\begin{pmatrix}
{\rho}-\bar{\rho}\\
{\rho}\bar{v}-\bar{\rho}\bar{v}
\end{pmatrix}\\
&=\frac{1}{2}{\rho}\abs{{v}-\bar{v}}^2+P({\rho}|\bar{\rho}).
\end{align*}
To check the one-sided condition, let
\[
\xi = (\xi_\rho, \xi_v),\,\,\bar \xi = (\bar{\xi}_\rho, \bar{\xi}_v) \in \re^+\times\re^d
\]
and $\bar U = (\bar \rho,\bar v)$ to find that
\begin{equation*}
\pa_k G(\bar U)\cdot F_k(\xi|\bar{\xi})= \xi_\rho \nabla_x \bar{v}: (\xi_v - \bar{\xi}_v)\otimes (\xi_v - \bar{\xi}_v)+ \left(p(\xi_\rho) -p(\bar{\xi}_\rho)-(\xi_\rho - \bar{\xi}_\rho)p^{\p}(\bar{\xi}_\rho)\right)\mbox{div}\,\bar{v}.
\end{equation*}
In \cite{FGJ}, the assumed one-sided condition was
\begin{equation}\label{eq:FGJ_condition}
z_k\pa_k v\cdot z\geq D(t)\abs{z}^2,\mbox{ for all }z\in \re^d,
\end{equation}
for some $D\in L^1((0,T))$. Note that in \eqref{eq:FGJ_condition}, we may choose $z = \xi_v - \bar{\xi}_v$, as well as $z = e_k$ to deduce that
\begin{equation*}
\pa_k G(\bar U)\cdot F_k(\xi |\bar{\xi})\gtrsim b(t)\big({\xi_\rho}\abs{\xi_{v}-\bar{\xi}_{v}}^2+ P(\xi_{\rho}|\bar{\xi}_{\rho})\big) =  b(t)H(\xi|\bar{\xi}),
\end{equation*}
for an appropriate $b \in L^1((0,T))$. This recovers the result of \cite{FGJ}.


\subsection{Elasticity}

In this section we consider the system of elasticity where, for homogeneous materials and in the absence of external forces, the balance of linear momentum takes the form
\begin{equation}\label{2ndOrderPDE}
\pa_t^2 y=\dv_x \Sigma(\nabla_xy).
\end{equation}
Here $y: Q\times(0,T)\rr\R^d$ denotes the deformation and $\Sigma$ is the Piola-Kirchhoff stress tensor which depends on the deformation gradient, but not $y$ or $\pa_t y$, as a consequence of frame-indifference. Letting $F:=\nabla y$ and ${v}=\pa_t y$, \eqref{2ndOrderPDE} can be written as the following system of conservation laws: 
\begin{equation}
\begin{array}{rll}
\pa_t {v}&=&\dv \Sigma(F),\\
\pa_t F&=&\nabla {v},
\end{array}\mbox{ for }(x,t)\in Q\times(0,T).\label{HE2}
\end{equation} 
Henceforth, we impose the assumption of hyperelasticity, i.e. that 
\[
\Sigma(F) = D_F W(F) = \left(\frac{\pa W}{\pa F_{i\alpha}}\right)_{i\alpha},
\]
where $W:\re^{d\times d} \to \re$ is the stored-energy function and we have adopted the convention of using greek and latin indices, respectively, for variables in the reference and deformed configurations. We note that system \eqref{HE2} can be written in the form
\[
\pa_t U + \pa_\alpha f_\alpha(U) = 0,
\]
where, writing $\{e_i\}_{1\leq i\leq d+d^2}$ for the standard basis of $\R^{d+d^2}$,
\[
U = \left(\begin{array}{c} v \\ F \end{array}\right) = {v}_i e_i+F_{i\alpha}e_{\alpha+di}\mbox{ and }f_\alpha(U) = \Sigma_{i\alpha}(F)e_i+{v}_ie_{\alpha + di}.
\]
Moreover, system \eqref{HE2} is endowed with the entropy-entropy flux pair
\[
\eta(U) = \eta(v,F) = \frac{1}{2} |v|^2 + W(F)\mbox{ and }q_\alpha(U) = q_\alpha(v,F) =  v_{i}\Sigma_{i\alpha}(F).
\]
We note that often the condition that $W(F)\to\infty$, as $\det F\to 0^+$ and $W(F) \equiv \infty$ if $\det F \leq 0$ is regarded as a physical requirement to exclude the interpenetration of matter. Then $\mathcal{O} = \{F:\det F > 0\}$ which for $d>1$ is a nonconvex set and gives rise to several open problems in the mathematical treatment of elasticity. Thus, we do not impose such assumptions and generally consider $\mathcal{O} = \R^d$. We refer the reader to \S \ref{sec:polyconvex} for further comments as well as to \cite{Ball_open_problems}.


\subsubsection{Convex elasticity and the one-dimensional case}

If the stored-energy function $W\in C^2$ is assumed strongly convex, system \eqref{HE2} fits into the present setting by imposing a coercivity and growth assumption of the form
\[
-1 + |F|^{p_1} \lesssim W(F) \lesssim 1 + |F|^{p_1},\quad p_1\geq 2,
\]
Indeed, \descref{H1}{(H1)} is then satisfied with $p = 2$. Moreover, by the assumed growth and convexity (in fact separate convexity suffices, see \cite[Proposition 2.32]{Dacorogna}), it follows that
\[
|DW(F)| \lesssim 1 + |F|^{{p_1}-1} \lesssim 1 + |F|^{p_1} \lesssim 1 + W(F).
\]
In particular, $|f_\al(v, F)| + |(v,F)|\lesssim 1 + \eta(v,F)$ which ensures \descref{H2a}{(H2a)} and \descref{H2b}{(H2b)}.

Next, note that
\[
D\eta(v,F) = \left(v, \Sigma(F)\right)^T
\]
whereas with an abuse of notation
\[
f_\alpha(v,F|\bar{v},\bar{F}) = \left(\Sigma_{i\alpha}(F|\bar{F}), 0\right)^T = \Sigma_{i\alpha}(F|\bar{F})e_i,
\]
where $\Sigma(F|\bar{F}) = \Sigma(F) - \Sigma(\bar{F}) - D\Sigma(\bar{F})(F-\bar{F})$. Then, $\left(f_\alpha(v,F|\bar{v},\bar{F})\right)_i = 0$ for all $i$ such that $\left(D\eta\right)_i$ is nonlinear, implying \descref{H2c}{(H2c)}. Hence, Theorem \ref{theorem1} applies to dissipative solutions
\[
(\bar{v},\bar{F}) \in L^1(\delta, T; B^{\alpha,\f}_q(\Dom)),\quad q\geq 4.
\]
Also, letting $\xi = (\xi_v,\xi_F)$, $\bar{\xi} = (\bar{\xi}_v,\bar{\xi}_F) \in\R^{d}\times\R^{d\times d}$, the one-sided condition \eqref{ineq:general} becomes
\begin{equation}\label{condition:H-E}
\left(\pa_\alpha \bar{v}_i\right) \Sigma_{i\alpha}(\xi_F|\bar{\xi}_F) + b(t) W(\xi_F|\bar{\xi}_F) \geq 0.
\end{equation}
Note that the above condition does not depend on $\bar{F}$. 

\paragraph{The case $d=1$} If $d=1$ system \eqref{HE2} is similar to the $p$-system and when $\Sigma^{\p\p}\neq0$ it becomes strictly hyperbolic with both characteristic fields genuinely nonlinear. It is then known that a shock-free solution to the Riemann problem satisfies $-sgn(\Sigma^{\p\p})\pa_x v(t,x)\geq0$ for a.e. $x\in\re$ and $t>0$ (see \cite{ChHs,ChFrLi,Smo}). Therefore it also satisfies \eqref{ineq:general} with $b\equiv0$ provided that the solution remains in the region $\mathcal{O}\subset\{\Sigma^{\p\p}\neq0\}$.

Next, suppose that $(v,F)$ is any Lipschitz solution to \eqref{HE2} which is self-similar, i.e. assume that 
\begin{equation*}
({v},F)(t,x)=(V(x/t),\mathcal{F}(x/t))
\end{equation*}
solves the following system:
\begin{align}
\zeta V^{\p}(\zeta)+\Sigma^{\p}(\mathcal{F}(\zeta))\mathcal{F}^{\p}(\zeta)&=0,\label{eqn:sshe1}\\
\zeta \mathcal{F}^{\p}(\zeta)+V^{\p}(\zeta)&=0.\label{eqn:sshe2}
\end{align}
Combining \eqref{eqn:sshe1} and \eqref{eqn:sshe2} we have
\begin{equation*}
\zeta^2 \mathcal{F}^{\p}(\zeta)=\Sigma^{\p}(\mathcal{F}(\zeta))\mathcal{F}^{\p}(\zeta).
\end{equation*}
If $\mathcal{F}^{\p}\neq0$ we find that $\zeta^2=\Sigma^{\p}(\mathcal{F}(\zeta))$ and, differentiating with respect to $\zeta$, that
\begin{equation*}
2\zeta=\Sigma^{\p\p}(\mathcal{F}(\zeta))\mathcal{F}^{\p}(\zeta).
\end{equation*} 
Hence, if $W$ is convex, i.e. $\Sigma^\p > 0$, we deduce that
\begin{equation*}
-\zeta^2\Sigma^{\p\p}(\mathcal{F}(\zeta))V^{\p}(\zeta)=\zeta\frac{2\zeta}{\mathcal{F}^{\p}(\zeta)}\Sigma^{\p}(\mathcal{F}(\zeta))\mathcal{F}^{\p}(\zeta)=2\zeta^2 \Sigma^{\p}(\mathcal{F}(\zeta))\geq0.
\end{equation*}
Therefore, in one space-dimension, any self-similar Lipschitz solution satisfies \eqref{ineq:general} with $b\equiv0$ and $\mathcal{O}\subset\{\Sigma^{\p\p}\neq0\}$.


\subsubsection{Polyconvex elasticity}\label{sec:polyconvex}
We note that convexity of the stored-energy function $W$ is ruled out in elasticity as a consequence of frame-indifference, a physical invariance that is axiomatic in continuum mechanics\footnote{Similarly, in nonlinear theories of electromagnetism convexity can be ruled out due to Lorenz invariance \cite{Serre}.}.
Instead, motivated by the static theory, a natural convexity condition for $W$ in elasticity is quasiconvexity (in the sense of Morrey), see \cite{Dacorogna}. In particular, $W$ is then also rank-one convex which implies the symmetrisability of system \eqref{HE2}. These are conditions strictly weaker than convexity and they become appropriate due to the existence of involutions for the system of elasticity. We refer the reader to \cite{Daf86, KS19, KV} for investigations on weak-strong uniqueness results for elasticity and general systems admitting involutions under these relaxed convexity assumptions. We note that the required regularity on the strong solution in these works is inconsistent with the requirements in the present article. 

Another convexity condition that arises naturally in the context of elasticity is polyconvexity which, in the case $d=3$, amounts to the existence of a convex function $\mathcal{G}:\R^{d\times d}\times\R^{d\times d}\times\R$ such that
\[
W(F)=\mathcal{G}(F,{\rm cof}(F),\det(F)).
\]
Indeed, polyconvex energies describe many physical models in elasticity, it is stronger that quasiconvexity, yet weaker than convexity, and allows for an existence theory in statics even under the mathematically challenging assumption that $W(F)\to\f$, as $\det F \to 0^+$, see \cite{Ball76}. The dynamic equations also admit a good theory for polyconvex energies and we refer the reader to \cite{tzavaras,tzavaras-weak-strong}. In particular, the polyconvex theory in dynamics finds its origins in the observation of Qin \cite{Qin} that smooth solutions of \eqref{HE2} satisfy the additional conservation laws
\begin{equation}\label{eq:Qin}
\begin{aligned}
\pa_t \det F & = \pa_\alpha\left(({\rm cof}F)_{i\alpha}v_i\right),\\
\pa_t ({\rm cof}F)_{k\gamma} & = \pa_\alpha\left(\e_{ijk}\e_{\alpha\beta\gamma}F_{j\beta}v_i\right).
\end{aligned}
\end{equation}
The validity of \eqref{eq:Qin} for $F = \nabla y$ is due to the fact that the minors are null-Lagrangians. Following the notation of \cite{tzavaras}, we find that system \eqref{HE2} can be embedded into the enlarged system
\begin{align}
\pa_t v_i&=\pa_{\al}(g_{i\al}(F,Z,w;F)),\label{HE-poly1}\\
\pa_t F_{i\al}&=\pa_{\al}v_i,\label{HE-poly2}\\
\pa_t Z_{k\ga}&=\pa_{\al}\left(\epsilon_{ijk}\epsilon_{\al\B\ga}F_{j\B}v_i\right),\label{HE-poly3}\\
\pa_tw&=\pa_{\al}(cof(F)_{i\al}v_i),\label{HE-poly4}
\end{align}
where $g_{i\al}$ is defined as
\begin{equation*}
g_{i\al}(F,Z,w;\tilde{F})=\frac{\pa \mathcal{G}}{\pa F_{i\al}}(F,Z,w)+\frac{\pa \mathcal{G}}{\pa Z_{k\ga}}(F,Z,w)\epsilon_{ijk}\epsilon_{\al\B\ga}\tilde{F}_{j\B}+(cof (\tilde{F}))_{i\al}\frac{\pa \mathcal{G}}{\pa w}(F,Z,w).
\end{equation*}
Indeed, the embedding of elasticity in the above system relies on the fact that the minors are themselves involutions of \eqref{HE-poly1}--\eqref{HE-poly4} in the sense that if at the initial time the augmented variables $(F,Z,w)$ are given by $(F,{\rm cof}F,\det F)$, then the same holds for all subsequent times.

For a strictly polyconvex $W$, system \eqref{HE-poly1}--\eqref{HE-poly4} falls into the present setting as it can be expressed in the form
\[
\pa_t U + \pa_\alpha f_\alpha(U) = 0,
\]
where, letting $e_i$ denote the standard basis in $\R^{22}$,
\begin{align*}
U&= \left(v,F,Z,w\right)^T = v_ie_i+F_{i\alpha}e_{\alpha + 3i}+Z_{k\ga}e_{9+3k+\ga}+we_{22},\\
f_\al(U)&=g_{i\al}(F,Z,w;F)e_{i}+v_{i}e_{\al + 3i}+\epsilon_{ijk}\epsilon_{\al\B\ga}F_{j\B}v_ie_{9+3k+\ga}+ {\rm cof}(F)_{i\al}v_ie_{22}
\end{align*}
and we recall that
\[
\left({\rm cof}F\right)_{i\alpha} = \frac12 \e_{ijk}\e_{\al\beta\ga} F_{j\beta}F_{k\ga}.
\]
Moreover, system \eqref{HE-poly1}--\eqref{HE-poly4} is endowed with the strictly convex entropy
\begin{equation*}
\eta(v,F,Z,w)=\frac{1}{2}\abs{v}^2+\mathcal{G}(F,Z,w).
\end{equation*}
In accordance with \cite{tzavaras,tzavaras-weak-strong}, we assume that $\mathcal{G}\in C^2$ satisfies
\begin{equation}\label{eq:growth_polyconvex}
-1 + |F|^{p_1} + |Z|^{p_2} + |w|^{p_3} \lesssim \mathcal{G}(F,Z,w) \lesssim 1 + |F|^{p_1} + |Z|^{p_2} + |w|^{p_3}, \quad p_1 > 4, \,\,p_2,\,p_3\geq 2
\end{equation}
and
\begin{equation}\label{eq:growth_D_polyconvex}
\abs{\frac{\pa \mathcal{G}}{\pa F}}+\abs{\frac{\pa \mathcal{G}}{\pa Z}}^{\frac{p_1}{p_1-1}}+\abs{\frac{\pa \mathcal{G}}{\pa w}}^{\frac{p_1}{p_1-2}}\lesssim 1+ \abs{F}^{p_1}+\abs{Z}^{p_2}+\abs{w}^{p_3}.
\end{equation}
We remark that in \cite{tzavaras,tzavaras-weak-strong} the requirement that $p_1>4$ relates to the validity of the weak continuity of minors in a Sobolev regularity setting. Next, note that \eqref{eq:growth_polyconvex} and \eqref{eq:growth_D_polyconvex} ensure \descref{H1}{(H1)}, \descref{H2a}{(H2a)}, and \descref{H2b}{(H2b)}. Indeed,
\begin{align*}
|f_\al(v,F,Z,w)| & \lesssim |g_{i\al}(F,Z,w)| + |v| + |F||v| + |F|^2|v|  \\
& \lesssim 1+ |v|^2 + |F|^4 + \abs{\frac{\pa \mathcal{G}}{\pa F}}+\abs{\frac{\pa \mathcal{G}}{\pa Z}} |F|+\abs{\frac{\pa \mathcal{G}}{\pa w}} |F|^2\\
& \lesssim 1+ |v|^2 + |F|^{p_1} + \abs{\frac{\pa \mathcal{G}}{\pa F}}+\abs{\frac{\pa \mathcal{G}}{\pa Z}}^{\frac{p_1}{p_1-1}}+\abs{\frac{\pa \mathcal{G}}{\pa w}}^{\frac{p_1}{p_1-2}} \\
& \lesssim 1 + \eta(v,F,Z,w).
\end{align*}
In the convex setting, the property that $|DW(F)| \lesssim 1 + |F|^{r-1}$ for separately convex functions with $r$-growth allowed us to fulfil \descref{H2b}{(H2b)}. Regarding \descref{H2c}{(H2c)},  we observe that
\begin{equation*}
G = D\eta=\left(v,\frac{\pa \mathcal{G}}{\pa{F}},\frac{\pa \mathcal{G}}{\pa Z},\frac{\pa \mathcal{G}}{\pa w}\right)
\end{equation*}
and $(D\eta)_i$ is linear for $i=1,2,3$. Moreover, a tedious computation shows that
\begin{align*}
f_\al(U|\bar{U})&=g_{i\al}(F,Z,w;F|\bar{F},\bar{Z},\bar{w};\bar{F})e_{i}+\epsilon_{ijk}\epsilon_{\al\B\ga}(F_{j\B}-\bar{F}_{j\B})(v_{i}-\bar{v}_{i})e_{9+3k+\ga}\nonumber\\
&\quad + \epsilon_{ijk}\epsilon_{\al\B\ga}(F_{j\B}-\bar{F}_{j\B})(F_{k\ga}-\bar{F}_{k\ga})v_i e_{22}
\end{align*}
and thus, for $i\geq 4$, $f_\al(U|\bar{U})_i \not\equiv 0$. However, for $L \leq 2p_1/(p_1 + 4)$, we find that $L \in (1,2)$ (as $p_1 >4$) and by Young's inequality we compute that
\begin{align*}
|f_\al(U)_i|^{L} & \lesssim 1+ |v|^{2} + \left(|F||v|\right)^{\frac{2p_1}{p_1+4}} + \left(|F|^2|v|\right)^{\frac{2p_1}{p_1+4}} \\
& \lesssim 1+ |v|^{2} + \left(|F|^{2\frac{p_1+4}{4}} + |v|^{\frac{p_1 + 4}{p_1}}\right)^{\frac{2p_1}{p_1+4}} \\
& \lesssim 1+ |v|^{2} + |F|^{p_1} \lesssim 1 + \eta(v,F,Z,w),
\end{align*}
which is \descref{H2c}{(H2c)}. Hence, Theorem \ref{theorem1} applies with $q\geq \max\{4,2p_1/(p_1-4)\}$ and letting $\xi = (\xi_v,\xi_F,\xi_Z,\xi_w)$, and respectively for $\bar{\xi}$, the one-sided condition becomes
\begin{align}
\pa_\al D\eta(\bar{U})\cdot f_{\al}(\xi|\bar{\xi}) & =g_{i\al}(\xi_F,\xi_Z,\xi_w;\xi_F|\bar{\xi}_F,\bar{\xi}_Z,\bar{\xi}_w;\bar{\xi}_F)\pa_\al\bar{v}_i \nonumber\\
&\quad +\epsilon_{ijk}\epsilon_{\al\B\ga}\left(\left(\xi_F\right)_{j\B}-\left(\bar{\xi}_F\right)_{j\B}\right)\left\{\left(\left({\xi}_v\right)_{i}-\left(\bar{\xi}_v\right)_{i}\right)\pa_\al\left(\frac{\pa \mathcal{G}}{\pa Z_{k\ga}}(\bar{U})\right) \right. \nonumber\\
&\quad\quad \left. + \left(\left(\xi_F\right)_{k\ga}-\left(\bar{\xi}_F\right)_{k\ga}\right)\left(\xi_v\right)_i \pa_\al\left(\frac{\pa \mathcal{G}}{\pa w}(\bar{U})\right) \right\}.\label{condition:H-E-poly}
\end{align}
Note that unlike the convex case \eqref{condition:H-E}, the condition for polyconvex elasticity also depends on $\bar{F}$.

\subsection{Shallow water magnetohydrodynamics}
We next consider the system for shallow water magnetohydrodynamics \cite{GKS} taking the form
\begin{align}
\pa_th+\dv_x(h{v})&=0,\label{eq:SWM1}\\
\pa_t(h{v})+\dv_x(h{v}\otimes{v}-h{b}\otimes{b})+\nabla_x(g h^2/2)&=0,\label{eq:SWM2}\\
\pa_t(h{b})+\dv_x(h{b}\otimes{v}-h{v}\otimes{b}) &=0,\label{eq:SWM3}
\end{align}
where $g>0$ is the gravitational constant. In the above system, $h$ and $v$ denote the thickness and velocity of the fluid respectively, and $b$ denotes the magnetic field. Note that typically the system for shallow water magnetohydrodynamics is presented by adding the term $v\dv_x(hb)$ on the left-hand side of \eqref{eq:SWM3}. However, if $\dv_x(hb)$ vanishes at the initial time, it remains zero and in accordance with \cite{GKS} we choose to work with system \eqref{eq:SWM1}--\eqref{eq:SWM3}. We note that the above equations fit the general framework \eqref{eqn:conlawgen} with
\begin{equation}\label{def:SWM:A-F}
U = \left(\begin{array}{c}h \\ v \\b \end{array}\right),\quad A(U) = \left(\begin{array}{c} h \\ hv\\ hb \end{array}\right),\quad F_k(U) = \left(\begin{array}{c} h v_k \\ h v v_k-hbb_k + (g h^2/2) e_k\\
hbv_k - hvb_k  \end{array}\right).
\end{equation}
The thickness $h$ is required to be positive and thus
\[
\mathcal{O} = \left\{(h,v,b)\in \re\times \re^d\times\R^d\,:\, h > 0\right\}.
\]
For system \eqref{eq:SWM1}--\eqref{eq:SWM3} we may choose $G$ and $H$ as  
\begin{equation}
G(U)=\left(g h-\frac{\abs{v}^2}{2}-\frac{\abs{b}^2}{2},v^{T}, b^{T}\right)\mbox{ and }H(U)=\frac{g}{2}h^2+\frac{1}{2}h\abs{v}^2+\frac{1}{2}h\abs{b}^2.
\end{equation}
By a similar calculation as in  \S \ref{sec:isnE} we get
\begin{align*}
DA(U)&=\begin{pmatrix}
1&{0}^T&{0}^T\\
v&h\mathbb{I}_d&\textbf{0}_d\\
b&\textbf{0}_d&h\mathbb{I}_d\end{pmatrix},\,
DA(U)^{-1}=\frac{1}{h}\begin{pmatrix}
h&{0}^T&{0}^T\\
-v&\mathbb{I}_d& \textbf{0}_d\\
-b&\textbf{0}_d&\mathbb{I}_d
\end{pmatrix},\\
DF_k(U)&=\begin{pmatrix}
v_k&h e_k^T&{0}^T\\
vv_k-bb_k+g he_k&h v_k\mathbb{I}_d+h v\otimes e_k&-h b_k\mathbb{I}_d-h b\otimes e_k\\
bv_k - vb_k& h b\otimes e_k - h b_k\mathbb{I}_d&h v_k\mathbb{I}_d - h v\otimes e_k\\
\end{pmatrix}.
\end{align*}
Next, we observe that
\begin{align*}
G(U)D^2A(U)&=\left(g h-\frac{1}{2}\abs{v}^2-\frac{1}{2}\abs{b}^2\right)\begin{pmatrix}
0&{0}^T&{0}^T\\
0&\textbf{0}_d&\textbf{0}_d\\
0&\textbf{0}_d&\textbf{0}_d
\end{pmatrix}+v_k
\begin{pmatrix}
0&e_k^{T}&{0}^T\\
e_k&\textbf{0}_d&\textbf{0}_d\\
{0}&\textbf{0}_d&\textbf{0}_d
\end{pmatrix}+b_k
\begin{pmatrix}
0&{0}^T&e_k^{T}\\
{0}&\textbf{0}_d&\textbf{0}_d\\
e_k&\textbf{0}_d&\textbf{0}_d
\end{pmatrix}
\nonumber\\
&=\begin{pmatrix}
0& v^{T}&b^{T}\\
v&\textbf{0}_d&\textbf{0}_d\\
b&\textbf{0}_d&\textbf{0}_d
\end{pmatrix}
\end{align*}
so that
\begin{equation*}
D^2H-G(U)D^2A(U)=\begin{pmatrix}
g&{0}^T&{0}^T\\
{0}&h\mathbb{I}_d&\textbf{0}_d\\
{0}&\textbf{0}&h\mathbb{I}_d 
\end{pmatrix}>0
\end{equation*}
in $\mathcal{O}$ and \eqref{eq:convexity gen} is satisfied. Let $p = 4/3$. By a similar argument as in \S\ref{sec:isnE} we obtain
\begin{align}\label{cal1:swm}
\abs{A(U)}^{\frac43} & \lesssim 1 + h^{2}+h^{2/3}\left(h^{2/3}\abs{v}^{4/3}\right)+h^{2/3}\left(h^{2/3}\abs{v}^{4/3}\right)\nonumber\\
& \lesssim 1+ h^{2}+h\abs{v}^2+h\abs{b}^2\lesssim 1+ H(U).
\end{align}
Moreover,
\begin{align*}
|F_k(U)| & \lesssim h^2 + h^{1/2}\left(h^{1/2}|v|\right) + h|v|^2 + h|b|^2 + \left(h^{1/2}|v|\right)\left(h^{1/2}|b|\right) \\
& \lesssim 1 + h^2 + h|v|^2 + h|b|^2 \lesssim 1 + H(U)
\end{align*}
so that \descref{H1}{(H1)}, \descref{H2a}{(H2a)}, and \descref{H2b}{(H2b)} are satisfied. See also \cite{GKS} for a similar analysis. Next, note that $G_1$ is the only nonlinear component of $G$. On the other hand, it is a matter of a calculation similar to \S\ref{sec:isnE} to verify that
\begin{align*}
F_k({U}|\bar{U})&=F_k({U})-F_k(\bar{U})-DF_k(\bar{U})DA(\bar{U})^{-1}(A({U})-A(\bar{U}))\\
&=\begin{pmatrix}
0\\
{h} ({v}-\bar{v})( {v}_k-\bar{v}_k)-{h}({b}-\bar{b})( {b}_k-\bar{b}_k)\\
{h}({b}-\bar{b})( {v}_k-\bar{v}_k) - {h} ({v}-\bar{v})( {b}_k-\bar{b}_k)
\end{pmatrix}+\frac{g}{2}
\begin{pmatrix}
0\\
({h}-\bar{h})^2e_k\\
0
\end{pmatrix}.
\end{align*}
In particular, $\left(F_k(\xi|\bar{\xi})\right)_1 \equiv 0$ and \descref{H2c}{(H2c)} also holds. We may thus apply Theorem \ref{theorem1} with $q \geq 8$. Regarding condition \eqref{ineq:general}, note that
\begin{align*}
H({U}|\bar{U})&=H({U})-H(\bar{U})-G(\bar{U})(A({U})-A(\bar{U}))\\
&=\frac{1}{2}{h}\abs{{v}-\bar{v}}^2+\frac{1}{2}{h}\abs{{b}-\bar{b}}^2+\frac{g}{2}\abs{{h}-\bar{h}}^2.
\end{align*}
Letting $\xi = (\xi_h,\xi_v,\xi_b)$, $\bar{\xi} = (\bar{\xi}_h, \bar{\xi}_v, \bar{\xi}_b) \in \R\times \R^d\times \R^d$, we then find that
\begin{align*}
\pa_k G(\bar{U})\cdot F_k(\xi|\bar{\xi})&= \xi_h \nabla_x \bar{v}: \left((\xi_v - \bar{\xi}_v)\otimes (\xi_v - \bar{\xi}_v)-(\xi_b - \bar{\xi}_b)\otimes (\xi_b - \bar{\xi}_b)\right)+\frac{g}{2}\abs{\xi_h - \bar{\xi}_h}^2\mbox{div}_x\bar{v}\nonumber\\
&+\xi_{h} \nabla_x \bar{b}: \left((\xi_b - \bar{\xi}_b)\otimes (\xi_v - \bar{\xi}_v) - (\xi_v - \bar{\xi}_v)\otimes (\xi_b - \bar{\xi}_b)\right).
\end{align*}
Note that condition \eqref{ineq:general} for the system of shallow water magnetohydrodynamics does not depend on the (distributional) derivative 
 of the thickness $h$.
 

\section{A nontrivial example and 1-D to multi-D extensions}\label{sec:triangular}
 In the previous examples, we looked at systems of the form \eqref{eqn:conlawgen} for which Theorem \ref{theorem1} can be applied and we investigated condition \eqref{ineq:general} for these systems. However, we are unable to construct explicit solutions satisfying the required regularity assumptions, yet fail to be Lipschitz. Note that Lipschitz functions immediately satisfy condition \eqref{ineq:general}. In the present section, we aim to find examples of solutions to systems of the form \eqref{eqn:conlawgen} that are merely in the H\"older space $C^{0,\B}$ but indeed satisfy the one-sided condition \eqref{ineq:general} and are thus unique by Theorem \ref{theorem1}.
 
 Indeed, in \S \ref{sec:triangular1}, we focus on a one-dimensional triangular system motivated by the study of multi-component chromatography and studied in \cite{Boris}. More precisely, based on the backward algorithm found in \cite{AGV-exact}, we provide a family of nontrivial $C^{0,\B}$ solutions satisfying \eqref{ineq:general} for a class of these triangular systems. Moreover, in \S \ref{sec:triangular2}, we propose a more restrictive form of system \eqref{eqn:conlawgen} which allows to extend 1-D to multi-D solutions. This way we obtain non-trivial states satisfying \eqref{ineq:general} for a multi-dimensional system. The proposed form is satisfied by the isentropic Euler system, shallow water magnetohydrodynamics, and the triangular system studied below.
 
 
  \subsection{Triangular system}\label{sec:triangular1}
Here, we construct a class of non-trivial solutions to the 1-D triangular system which reads as follows:
 \begin{equation}\label{sys:tri}
 \begin{array}{rl}
 \pa_t u+\pa_xf(u)&=0,\\
 \pa_tv+\pa_x(g(u)v)&=0,
 \end{array}\mbox{ for }(x,t)\in\Dom\times\R_+,
 \end{equation}
 where $f,\,g\in C^2(\R)$ and $f$ is strictly convex. System \eqref{sys:tri} is endowed with the smooth entropies
 \begin{equation}\label{def:tri-ent}
 \eta(u,v)=\psi(u)+e^{-\phi(u)}k(ve^{\phi(u)})
 \end{equation}
 where $\phi$ is the primitive of the function $u\mapsto \frac{g^{\p}(u)}{g(u)-f^{\p}(u)}$ which must be assumed integrable. We can check that $\eta(u,v)$ becomes strictly convex if $\psi$, $k$ are convex functions and $u\mapsto e^{-\phi(u)}$ is concave. For a detailed discussion we refer the reader to \cite{Boris}. 
 
 Henceforth, we assume that $g=h\circ f^{\p}$ for some $C^2$ function $h$ and we wish to investigate the conditions of Theorem \ref{theorem1} and \eqref{ineq:general} in particular. For the sake of simplicity, let us consider the function $f(u)=(q+1)^{-1}\abs{u}^{q+1}$ for $q\in[1,2)$, and we will later generalise the construction in \S \ref{sec:triangular1.1}. It is clear that $f\in C^2(\R)$ is strictly convex and $f^{\p}(u)=u\abs{u}^{q-1}$. Subsequently, $u(\cdot,t)\in C^{0,1/q}$ for each $t>0$ provided that $x\mapsto f^{\p}(u(x,t))$ is Lipschitz for $t>0$, and hence $u(\cdot,t)\in B^{\al,\f}_l([-r,r])$ with $\al=q^{-1}$ and for all $1\leq l\leq \f$, $r>0$. 
 
Note that system \eqref{sys:tri} can be written in the form of \eqref{eqn:conlaw} with $U=(u,v)$ and $F(U)=(f(u),g(u)v)$ where we abused notation and opted to use $F$ for the flux of \eqref{eqn:conlaw} and $f$ for the flux of the scalar conservation law in \eqref{sys:tri}. Moreover, observe that
 \begin{equation}
 F(U|\bar{U})=\left(f(u|\bar{u}),g(u|\bar{u})v+g^{\p}(\bar{u})(v-\bar{v})(u-\bar{u})\right).
 \end{equation}
{\color{black} Regarding the growth and coercivity conditions for system \eqref{sys:tri}, we instead work in an $L^\f$ setting. More precisely, we assume that the first components of both solutions (weak and Besov), belong to the class, $\{u:\|u\|_{L^\f(\Dom)}\leq M_1\}$ for some $M_1>0$ which can be ensured by choosing appropriate initial data. Then also $v\in L^\f$ and no coercivity or growth assumptions are required, see Remark \ref{rem:bounded}.}
In verifying the conditions of Theorem \ref{theorem1} we are free to choose any entropy from the family of entropies \eqref{def:tri-ent} and we make the special choice, $\psi(u)=u^{2}$. We take $h(s)=-\la s^{2m+1}$ for $m\in \N$ and $\la>0$, i.e. $g = -\lambda \left(f^\prime\right)^{2m+1}$. 
Then, as $f^{\p}(0)=0$, we infer that
 \begin{align}
 \phi(u)=\int\limits_{0}^{u}\frac{\la(2m+1)f^{\p\p}(s)(f^{\p}(s))^{2m-1}}{\la(f^{\p}(s))^{2m}+1}\,ds=\frac{2m+1}{2m}\log(1+\la(f^{\p}(u))^{2m})
 \end{align}
 and
 \begin{equation*}
 e^{-\phi(u)}=\left(1+\la(f^{\p}(u))^{2m}\right)^{-\frac{2m+1}{2m}}=\left(1+\la\abs{u}^{2mq}\right)^{-\frac{2m+1}{2m}}=:\mathcal{K}(u).
 \end{equation*}
 Subsequently, we may compute that
 \begin{align*}
 \mathcal{K}^{\p}(u)&=-\la\frac{(2m+1)qu\abs{u}^{2mq-2}}{\left(1+\la\abs{u}^{2mq}\right)^{\frac{4m+1}{2m}}},\\
 \mathcal{K}^{\p\p}(u)&=-\la\frac{(2m+1)(2mq-1)q\abs{u}^{2mq-2}}{\left(1+\la\abs{u}^{2mq}\right)^{\frac{4m+1}{2m}}}+\la^{2}\frac{(2m+1)(4m+1)q^2\abs{u}^{4mq-2}}{\left(1+\la\abs{u}^{2mq}\right)^{\frac{6m+1}{2m}}}\\
 &=-\frac{\la(2m+1)q\abs{u}^{2mq-2}\left[(2mq-1)-\la(2mq+q+1)\abs{u}^{2mq}\right]}{\left(1+\la\abs{u}^{2mq}\right)^{\frac{6m+1}{2m}}}.
 \end{align*}
We may thus choose $\la>0$ small enough such that $\mathcal{K}^{\p\p}\leq 0$ for all $\abs{u}\leq M_1$ and $\eta$ is strictly convex.

We proceed to construct the nontrivial solution satisfying \eqref{ineq:general} consisting of a rarefaction wave in the first component and a Lipschitz solution in the second. To this end, consider the scalar conservation law $\pa_t u+\pa_xf(u)=0$ for $(x,t)\in\R\times\R_+$ and, for some $x_0\in\R$, the Riemann data
 \begin{equation}
 u_0(x)=\left\{\begin{array}{rl}
 u_L&\mbox{if }x<x_0,\\
 u_R&\mbox{if }x>x_0.
 \end{array}\right.
 \end{equation}
 Note that, since $f$ is convex, for $u_L<u_R$ we get the following structure of $u$ for all $t>0$
 \begin{equation}\label{tri-sys:soln1}
 u(x,t)=\left\{\begin{array}{cl}
 u_L&\mbox{if }x\leq x_0+f^{\p}(u_L)t,\\
 (f^{\p})^{-1}\left(\frac{x-x_0}{t}\right)&\mbox{if }x_0+f^{\p}(u_L)t<x<x_0+f^{\p}(u_R)t,\\
 u_R&\mbox{if }x\geq x_0+f^{\p}(u_R)t.
 \end{array}\right.
 \end{equation}
Theorem \ref{theorem1} is stated for spatially periodic solutions and we next provide the appropriate periodic modification on $[-r,r]$ for some $r>0$. We first modify $u_0$ to obtain periodic data as
  \begin{equation}\label{data:periodic}
 \bar{u}_0(x)=\left\{\begin{array}{cl}
 u_L&\mbox{if }x<x_0,\\
 u_R&\mbox{if }x_0<x<y_0,\\
 u_R+\frac{x-y_0}{y_1-y_0}(u_L-u_R)&\mbox{if }y_0<x<y_1,\\
 u_L&\mbox{if }y_1<x,
 \end{array}\right.
 \end{equation}
 where $x_0<y_0<y_1$. Then, the corresponding entropy solution $\bar{u}$ to $\pa_tu+\pa_xf(u)=0$ has the following structure
 \begin{equation}\label{tri-sys:soln2}
 \bar{u}(x,t)=\left\{\begin{array}{cl}
 u_L&\mbox{if }x\leq x_0+f^{\p}(u_L)t,\\
 (f^{\p})^{-1}\left(\frac{x-x_0}{t}\right)&\mbox{if }x_0+f^{\p}(u_L)t<x<x_0+f^{\p}(u_R)t,\\
 u_R&\mbox{if } x_0+f^{\p}(u_R)t\leq x\leq y_0+tf^{\p}(u_R),\\
 \Theta(x,t)&\mbox{if } y_0+f^{\p}(u_R)t\leq x\leq y_1+tf^{\p}(u_L),\\
 u_L&\mbox{if }y_1+tf^{\p}(u_L)<x.
 \end{array}\right.
 \end{equation}
 Above $\Theta(x,t)$ is given by
 \begin{equation}
  \Theta(x,t)=u_R+\frac{z-y_0}{y_1-y_0}(u_L-u_R)\mbox{ for }x=z+tf^{\p}\left(u_R+\frac{z-y_0}{y_1-y_0}(u_L-u_R)\right),
 \end{equation}
 for $ y_0+f^{\p}(u_R)t\leq x\leq y_1+tf^{\p}(u_L),t\in[0,T]$. Note that for sufficiently large $y_1-y_0$, the function $x\mapsto\Theta(x,t)$ remains Lipschitz for $y_0+f^{\p}(u_R)t\leq x\leq y_1+tf^{\p}(u_L),t\in[0,T]$. Next, define $B_0$ as \begin{equation}\label{def:C2}
 B_0:=\sup_{t\in[0,T]}\left\{\sup\left\{\frac{\abs{\Theta(x_1,t)-\Theta(x_2,t)}}{\abs{x_2-x_1}}:x_2\neq x_1,\,x_1,\,x_2\in [y_0+f^{\p}(u_R)t,y_1+tf^{\p}(u_L)]\right\}\right\}.
 \end{equation}
 Let us fix a time $t_0>0$ and note that $(x,t)\mapsto h(f^{\p}(u(x,t)))$ is a Lipschitz function for $x\in\R$ and $t\in[t_0,T]$. Now we consider the data 
 \begin{equation}
 U_0(x)=(u(x,t_0),v_0(x))
 \end{equation}
 where $u(\cdot,\cdot)$ is as in \eqref{tri-sys:soln2} and $v_0\in Lip(\R)$ satisfies $v_0(x)=v_C$ when $x\leq x_0$ and $x\geq y_1$ for some constant $v_C$. For the above data the entropy solution of the first equation of \eqref{sys:tri} looks like 
 \begin{equation}
 U_1(x,t)=u(x,t+t_0),
 \end{equation}
 whereas the other component $U_2$ can be solved by the method of characteristics and $U_2$ remains Lipschitz for $t>0$, for instance see \cite{Boris}. Then $U=(U_1,U_2)$ solves the triangular system \eqref{sys:tri}.  We observe that there exists $r>0$ such that $U=(u_L,v_C)$ for all $x\in \R\setminus[x_0-r,y_1+r]$ and $t\in[0,T-t_0]$. Note that, with a suitable change of variables, $U$ can be transformed into a function $\tilde{U}:[-1,1]\times[0,T_1]\rr\R^2$ for some $T_1>0$ such that $\tilde{U}(-1,t)=\tilde{U}(1,t)$ for $t\in[0,T_1]$ and $\tilde{U}$ solves \eqref{sys:tri}. Therefore, without loss of generality, we check condition \eqref{ineq:general} for $U$. We observe that
 %
 \begin{equation*}
 (\pa_x{U_1},\pa_x{U_2})\cdot F((\xi_1,\xi_2)|(\bar{\xi}_1,\bar{\xi}_2))=\pa_x U_1 f(\xi_1|\bar{\xi}_1)+\pa_x U_2\left(g(\xi_1|\bar{\xi}_1)\xi_2+g^{\p}(\bar{\xi}_1)(\xi_2-\bar{\xi}_2)(\xi_1-\bar{\xi}_1)\right).
 \end{equation*}
 Note that $\pa_x U_1=\pa_xu\geq-B_0$ where $B_0$ is as in \eqref{def:C2}. Therefore, for $\xi$, $\bar\xi$ in a compact, we have
 \begin{equation*}
 \pa_x U\cdot F((\xi_1,\xi_2)|(\bar{\xi}_1,\bar{\xi}_2))\geq - C \left(\|U_2\|_{Lip(\R\times[0,T])}+B_0\right)(\abs{\xi_1-\bar{\xi}_1}^2+\abs{\xi_2-\bar{\xi}_2}^2),
  \end{equation*}
 where $C>0$ is a constant depending on the function $g$ and the compact set where $\xi$, $\bar{\xi}$ lie.
 
 
 \subsubsection{General states for the first component}
 \label{sec:triangular1.1}
 
We next wish to find more states for a class of general scalar conservation laws
 \begin{equation}\label{eqn:scalar}
 \pa_t u+\pa_x f(u)=0\mbox{ for }x\in\R\mbox{ and }t>0
 \end{equation}
which are H\"older continuous, satisfy condition \eqref{ineq:general} and can express the first component of the triangular system \eqref{sys:tri}.
 
 \begin{remark}
 Note that for \eqref{eqn:scalar} we work on $\R$ and data such that $u_0(x)=u_L$ for $x<-r$ and $u_0(x)=u_R$ for $x>r$. By a similar argument as in \eqref{data:periodic}, we can construct data $\bar{u}_0$ such that $\bar{u}_0(x)=u_L$ for $x>r_1>r$ and $\bar{u}_0=u_0$ on $[-r,r]$. From the previous observations, we know that the entropy solution $\bar{u}$ to \eqref{eqn:scalar} corresponding to data $\bar{u}_0$ is Lipschitz on $(\R\setminus[-r,r])\times[0,T]$. Therefore, it is with no loss in generality to work on $\R$ since the modification to a periodic solution preserves the H\"older regularity and the one-sided bound condition \eqref{ineq:general}.
 \end{remark}

 \begin{proposition}\label{prop:scalar}
 	For $T>0$, let $u\in C([0,T],L^1_{loc}(\R))\cap L^{\f}(\R\times[0,\f))$ be an entropy solution to \eqref{eqn:scalar} for a strictly convex flux $f\in C^2(\R)$. Suppose further that $u(\cdot,T)$ satisfies the regularity assumption: 
 	\begin{equation}\label{condition1}
 	f^{\p}(u(\cdot,T))\in C^{0,\B}(\R)\mbox{ for some }\B\in(0,1). 
 	\end{equation}
 	Then, for $t\in(0,T]$, $f^{\p}(u(\cdot,t))$ is H\"older continuous with 
 	\begin{equation*}
 	\abs{f^{\p}(u(\cdot,t))}_{C^{0,\B}([-M,M])}\leq \max\{\abs{f^{\p}(u(\cdot,T))}_{C^{0,\B}(\R)},(2M)^{1-\B}t^{-1}\}\mbox{ for }M>0.
 	\end{equation*} 
 	Moreover, if there exists a constant $B_1$ such that 
 	\begin{equation}\label{condition2}
 	\left(u(x-\De x,T)-u(x,T)\right)_+\leq B_1\De x\mbox{ for }x\in\R,\,\De x>0,
 	\end{equation}
 	then $u$ satisfies \eqref{ineq:general} with $b(t)=B_1$, that is,
 	\begin{equation*}
 	\pa_x u(\cdot,t)\geq -B_1\mbox{ in }\mathcal{D}^{\p}(\R)\mbox{ for all }t\in(0,T].
 	\end{equation*}

 \end{proposition}
 
 \begin{remark}\label{remark:scalar}
 	Let $h(\cdot)\in L^\f([-M,M])$ be a function such that the map $x\mapsto x-Tf^{\p}(h(x))$ is non-decreasing and right-continuous for $T>0$. Due to \cite{AGV-exact} we know that $h(x)$ is a reachable state from initial data for \eqref{eqn:scalar}.  By the backward algorithm \cite{AGV-exact}, we can show that $u(\cdot,T)$ satisfying \eqref{condition1} or \eqref{condition2} is achievable provided the map $x\mapsto x-Tf^{\p}(u(x,T))$ is non-decreasing. 
 \end{remark}
\begin{remark}
	Suppose we take $f(u)=(q+1)^{-1}\abs{u}^{q+1}$ and the map $x\mapsto f^{\p}(u(x,T))$ is $C^{0,\B}(\R)$ with $\B q^{-1}>1/2$. If $x\mapsto h(f^{\p}(u(x,T)))$ is Lipschitz then we can construct a solution, $U$ to \eqref{sys:tri} with data $U_0(x)=(u(x,t_0),v_0(x))$ with $t_0\in(0,T)$ where $u$ is as in Remark \ref{remark:scalar} and $v_0\in C^1(\R)$. By a similar argument we can then show that $U$ satisfies all the hypothesis of Theorem \ref{theorem1}. 
\end{remark}

\begin{proof}[Proof of Proposition \ref{prop:scalar}]
	 Since $f^{\p}$ is a strictly increasing function and $x\mapsto f^{\p}(u(x,T))$ is H\"older continuous, we have that $x\mapsto u(x,T)$ is continuous. 
	 Therefore, the maximal and the minimal backward characteristics coincide, see \cite{AGV-structure,Daf:Charac}, and thus, at each $(x,T)$ with $x\in\R$, there is only one genuine backward characteristic, say $\xi(x,T;t)$. Then, for $x\in\R$, $t\in(0,T]$ we find that
	\begin{align*}
	x&=\xi(x,T;t)+(T-t)f^{\p}(u(\xi(x,T;t),t))\\
	&=\xi(x,T;t)+(T-t)f^{\p}(u(x,T)).
	\end{align*}
	Let $C_0=\abs{f^{\p}(u(\cdot,T))}_{C^{0,\B}(\R)}$ and observe that, for $\abs{x}\leq M$, the map $x\mapsto \xi(x,T;t)$ is H\"older continuous for $t\in(0,T]$ as
	\begin{align*}
	\abs{\xi(x_1,T;t)-\xi(x_2,T;t)}&=\abs{x_1-f^{\p}(u(x_1,T))(T-t)-x_2+f^{\p}(u(x_2,T))(T-t)}\\
	&\leq \abs{x_1-x_2}+ T\abs{f^{\p}(u(x_1,T))-f^{\p}(u(x_2,T))}\\
	&\leq ((2M)^{1-\B}+C_0T)\abs{x_1-x_2}^\B.
	\end{align*}
	Hence, fixing $x_1<x_2$, the following two cases arise.
	\begin{enumerate}
		\item $\xi(x_2,T;0)-\xi(x_1,T;0)\geq x_2-x_1$: Then, we compute 
		\begin{align*}
		\abs{f^{\p}(u(\xi(x_1,T;t),t))-f^{\p}(u(\xi(x_2,T;t),t))}&=\abs{f^{\p}(u(x_1,T))-f^{\p}(u(x_2,T))}\\
		&\leq C_0\abs{x_1-x_2}^{\B}\\
		&\leq C_0\abs{\xi(x_1,T,t)-\xi(x_2,T;t)}^{\B}.
		\end{align*}
		
		\item $\xi(x_2,T,0)-\xi(x_1,T,0)< x_2-x_1$: Note that two backward characteristics (which are genuine in our case) cannot meet at time $t>0$. Hence, there exist $\xi_0\in\R$ and $\de\geq0$ such that
		\begin{equation*}
		\xi(x_1,T;t)=\xi_0+f^{\p}(u(x_1,T))(t+\de)\mbox{ and }\xi(x_2,T;t)=\xi_0+f^{\p}(u(x_2,T))(t+\de)
		\end{equation*}
		for all $t\in[0,T]$. Therefore, for $t\in(0,T]$, we have
		\begin{align*}
		\abs{f^{\p}(u(\xi(x_1,T;t),t))-f^{\p}(u(\xi(x_2,T;t),t))}&=\frac{\abs{\xi(x_1,T;t)-\xi(x_2,T;t)}}{t+\de}\\
		&\leq \frac{1}{t}\abs{\xi(x_1,T;t)-\xi(x_2,T;t)}.
		\end{align*}
	\end{enumerate}
	
	Note that for a fixed $t\in(0,T]$, the map $x\mapsto \xi(x,T;t)$ is continuous and strictly increasing, and hence a bijection between $\R\times\{T\}$ and $\R\times\{t\}$. Therefore, for $z_1,z_2\in \R$ there exist unique $x_1,x_2$ such that $\xi(x_1,T;t)=z_1$ and $\xi(x_2,T;t)=z_2$. By the previous observation, we thus have
	\begin{equation*}
	\abs{f^{\p}(u(z_1,t))-f^{\p}(u(z_2,t))}\leq \max\left\{C_0,\frac{(2M)^{1-\B}}{t}\right\}\abs{z_1-z_2}^{\B}\mbox{ for }z_1,z_2\in[-M,M],
	\end{equation*}
	which proves the H\"older regularity. Next, suppose that $u(x,T)$ satisfies the following
	\begin{equation*}
	\left(u(x-\De x,T)-u(x,T)\right)_+\leq B_1\De x\mbox{ for any }\De x>0.
	\end{equation*}
	Fix a $t\in(0,T]$. Suppose $u(z_1,t)>u(z_2,t)$ for some $z_1<z_2$. From the previous observation, there exist unique $x_1,x_2 $ such that $z_1=\xi(x_1,T;t)$ and $z_2=\xi(x_2,T;t)$. Subsequently, we have $u(z_1,t)=u(x_1,T)$ and $u(z_2,t)=u(x_2,T)$. From the increasing property of the map $x\mapsto \xi(x,T;t)$ we get $x_1<x_2$. Hence, $u(x_1,T)>u(x_2,T)$. Since $u\mapsto f^{\p}(u)$ is strictly increasing we have $f^{\p}(u(x_1,T))>f^{\p}(u(x_2,T))$ whereas we know that
	\begin{align*}
	\xi(x_2,T;t)-\xi(x_1,T;t)&=x_2-x_1-(T-t)f^{\p}(u(x_2,T))+(T-t)f^{\p}(u(x_1,T))\\
	&>x_2-x_1.
	\end{align*}
	Then we also find that
	\begin{align*}
	0\leq u(\xi(x_1,T;t),t)-u(\xi(x_2,T;t),t)
	&=u(x_1,T)-u(x_2,T)\\
	&\leq B_1(x_2-x_1)\\
	&\leq B_1( \xi(x_2,T;t)-\xi(x_1,T;t)).
	\end{align*}
	Therefore, $(u(z_1,t)-u(z_2,t))_+\leq B_1(z_2-z_1)$. Finally, for $\De x>0$, we infer that
	\begin{equation*}
	\frac{u(x+\De x,t)-u(x,t)}{\De x}\geq \left\{
	\begin{array}{rl}
	0,&\mbox{ if }u(x+\De x,t)>u(x,t),\\
	-B_1,&\mbox{ if }u(x+\De x,t)<u(x,t).
	\end{array}\right.
	\end{equation*}
	Now let $\varphi\in C_c^1(\R) $ such that $\varphi\geq0$. By a change of variables we find that
	\begin{equation}\label{cal1:scalar}
	-\int\limits_{\R}u(x,t)\frac{\varphi(x+\De x)-\varphi(x)}{\De x}\,dx=\int\limits_{\R}\frac{u(x,t)-u(x-\De x)}{\De x}\varphi(x)\,dx\geq -B_1\int\limits_{\R}\varphi(x)\,dx.
	\end{equation}
	Since $\varphi\in C^1_c(\R)$ we have 
	\begin{equation*}
	\int\limits_{\R}\abs{u(\cdot,t)}\frac{\abs{\varphi(x+\De x)-\varphi(x)}}{\De x}\,dx\leq \|u\|_{L^{\f}(\R\times[0,\f))}\|\varphi^{\p}\|_{L^{\f}(\R)}\mathcal{L}^1(supp(\varphi))
	\end{equation*}
	where $\mathcal{L}^1$ denotes the one-dimensional Lebesgue measure. Then, by dominated convergence, we may pass to the limit in \eqref{cal1:scalar} as $\De x\rr0$ to deduce
	\begin{equation*}
	-\int\limits_{\R}u(x,t)\varphi^{\p}(x)\,dx\geq -B_1\int\limits_{\R}\varphi(x)\,dx,
	\end{equation*}
	that is, $\pa_x u(\cdot,t)\geq -B_1$ in $\mathcal{D}^{\p}(\R)$ for all $t\in(0,T]$. 
\end{proof}


 \subsection{Multi-D planar extensions}\label{sec:triangular2}
 
 In this section, we wish to give a sufficient condition for system \eqref{eqn:conlawgen} to admit a planar extension of one-dimensional solutions. In particular, for these systems it will be enough to study solutions in 1-D and then extend them to multi-D by the procedure described below. We also verify that the isentropic Euler system, the equations of shallow water magnetohydrodynamics, as well as the triangular system satisfy the condition for planar extension. This way we can extend the class of states obtained for the triangular system to the multi-dimensional setting. For the study of multi-dimensional planar waves for the isentropic Euler system \eqref{isentropic_euler} we refer to \cite{ChFrLi,FeiKre}. 
 
 Let us consider a hyperbolic system in the following form:
 \begin{equation}
 \begin{array}{rl}
 \pa_t {A}_1(w)+\pa_1F^{A_1}_1(w)+\pa_{k}F^{A_1}_k(w,z)&=0,\\
 \pa_t {A}_2(w,z)+\pa_1F^{A_2}_1(w,z)+\pa_{k}F^{A_2}_k(w,z)&=0,
 \end{array}\mbox{ for }x=(x_1,\cdots,x_d)\in\Dom,\,t>0\label{sp:sys}
 \end{equation}
 where $F^{A_1}_1\in C^2(\R^n,\R^n),F^{A_1}_k\in C^2(\R^{n\times m},\R^n)$ for $2\leq k\leq d$ and $F^{A_2}_j\in C^2(\R^{n+m},\R^m)$ for $1\leq j\leq d$. Let $(G_1,H_1),(G_2,H_2)$ be determined by \eqref{eq:entropy-entropy_flux gen} corresponding to $A_1$ and $A_2$ respectively and impose the condition 
 \begin{equation}\label{condition:planar}
 F^{A_2}_1(w,0)={A}_2(w,0)=G_2(w,0)=0
 \end{equation} 
 which allows for the planar extension to multi-D. Indeed, suppose that $\bar{w}$ is a weak solution to $ \pa_t {A}_1(w)+\pa_1F^{A_1}_1(w)=0$ on $\Dom^1\times[0,T]$ with initial data $\bar{w}_0$ where $\Dom^1=[0,1]$ is the 1-D flat torus. Then we define multi-D initial data as
 \begin{align}\label{sp:data}
 w_0(x_1,\tilde{x}_2)=\bar{w}_0(x_1)\mbox{ and }z_0(x_1,\tilde{x}_2)=0\mbox{ for }x_1\in\Dom^1,\,\tilde{x}_2=(x_2,\cdots,x_d)\in\Dom^{d-1},
 \end{align}
 where $\Dom^{d-1}=[0,1]^{d-1}$ is the $(d-1)$-dimensional torus. Now we claim that 
 \[
 w(x_1,\tilde{x}_2)=\bar{w}(x_1),z(x_1,\tilde{x}_2)=0
 \] 
 is a weak solution to \eqref{sp:sys} with initial data $(w_0,z_0)$ as in \eqref{sp:data}. Since $F^{A_2}_1(w,0)={A}_2(w,0)=0$ and $w$ is independent of the $\tilde{x}_2$ variable, we find that
 \begin{equation}
 \pa_{k}F^{A_1}_k(w,z)=\pa_1F^{A_2}_1(w,z)=\pa_{k}F^{A_2}_k(w,z)=0\mbox{ for }2\leq k\leq d.
 \end{equation} 
 Hence, $(w,z)$ is a weak solution to system \eqref{sp:sys}. 
  
 We next claim that if the 1-D solution satisfies \eqref{ineq:general}, then also the multi-D extension satisfies the respective one-sided condition. Indeed, for the 1-D system $ \pa_t {A}_1(w)+\pa_1F^{A_1}_1(w)=0$ condition \eqref{ineq:general} becomes 
 \begin{equation}\label{ineq:gen-1D}
 \pa_1G_1(w)\cdot F_1^{A_1}(\xi_w|\bar{\xi}_w)+b_1(t)H_1(\xi_w|\bar{\xi}_w)\geq0.
 \end{equation}
 Suppose $\bar{w}$ satisfies \eqref{ineq:gen-1D}. Since $(w,z)$ is independent of $x_k$ for $2\leq k\leq d$, we have $\pa_k G_1(w,z)\cdot F_k^{A_1}(\xi|\bar{\xi})=0$ and $\pa_k G_2(w,z)\cdot F_k^{A_2}(\xi|\bar{\xi})=0$ for $2\leq k\leq d$ where $\xi=(\xi_w,\xi_z)$. We also observe that $\pa_1G_2(w,0)\cdot F_1^{A_2}(\xi|\bar{\xi})=0$. Therefore, $(w,z)$ also satisfies \eqref{ineq:general}.
 
 We note that all systems considered in the previous examples, apart from elasticity, can be written in the form \eqref{sp:sys}. For the isentropic Euler system \eqref{isentropic_euler} and $v=(v_1,\cdots,v_d)^T$, we set $w=(\rho,v_1)$, $z=(v_2,\cdots,v_d)$ and choose
 \begin{equation}
 {A}_1(\rho,v_1)=\begin{pmatrix}
 \rho\\
 \rho v_1
 \end{pmatrix}\mbox{ and }{A}_2(\rho,v)=\begin{pmatrix}
 \rho v_2\\
 \vdots\\
 \rho v_d
 \end{pmatrix}.
 \end{equation}
Then we can choose fluxes $\{F^{A_1}_j,F^{A_2}_j;1\leq j\leq d\}$ as
\begin{align}
&F^{A_1}_1(\rho,v)=\begin{pmatrix}
\rho v_1\\
\rho  v_1^2+p(\rho)
\end{pmatrix},\,
F^{A_1}_k(\rho,v)=\begin{pmatrix}
\rho v_k\\
\rho  v_1v_k
\end{pmatrix},\\
&F^{A_2}_1(\rho,v)=\begin{pmatrix}
	\rho v_2v_1\\
	\vdots\\	
	\rho  v_dv_1
\end{pmatrix},\,
F^{A_2}_k(\rho,v)=\begin{pmatrix}
	\rho v_2v_k\\
	\vdots\\
	\rho  v_dv_k
\end{pmatrix}+p(\rho)\begin{pmatrix}
 \de_{k2}\\
\vdots\\
\de_{kd}
\end{pmatrix}\mbox{ for }2\leq k\leq d,
\end{align} 
where $\de_{ij}$ is the Kr\"onecker delta. Note that \eqref{condition:planar} is satisfied.

Moreover, we observe that the system of shallow water magnetohydrodynamics can be represented in the form of \eqref{sp:sys} with $w=(h,v_1,b_1)$ and $z=(v_2,\cdots,v_d,b_2,\cdots,b_d)$. Now the choice for ${A}_1,{A}_2$ is the following
\begin{equation}
{A}_1(h,v_1,b_1)=\begin{pmatrix}
h\\
hv_1\\
hb_1
\end{pmatrix}\mbox{ and }{A}_2(h,v,b)=\begin{pmatrix}
h\tilde{v}_2\\
h\tilde{b}_2
\end{pmatrix}\mbox{ where }\tilde{v}_2=(v_2,\cdots,v_d)^T,\tilde{b}_2=(b_2,\cdots,b_d)^T.
\end{equation}
Fluxes $\{F^{A_1}_j,F^{A_1}_j;1\leq j\leq d\}$ can be chosen as follows
\begin{align*}
&F^{A_1}_1(h,v_1)=\begin{pmatrix}
hv_1\\
hv_1^2-hb_1^2+gh^{2}/2\\
hb_1v_1 - hv_1b_1   
\end{pmatrix},\,
F^{A_1}_k(h,v,b)=\begin{pmatrix}
hv_k\\
hv_1v_k-hb_1b_k\\
hb_1v_k - hv_1b_k   
\end{pmatrix},\\
&F^{A_2}_1(h,v,b)=\begin{pmatrix}
h \tilde{v}_2 v_1-h\tilde{b}_2b_1 \\
h\tilde{b}_2v_1- h\tilde{v}_2b_1   
\end{pmatrix},\,
F^{A_2}_k(h,v,b)=\begin{pmatrix}
h \tilde{v}_2 v_k-h\tilde{b}_2b_k + (g h^2/2) e_k\\
h\tilde{b}_2v_k - h\tilde{v}_2b_k 
\end{pmatrix}\mbox{ for }2\leq k\leq d.
\end{align*} 
Note that \eqref{condition:planar} is also satisfied in this case.


{\color{black}Lastly, to extend the triangular system in multi-D we can take $w=(u,v)$ and $A_1(w)=(u,v)$ with $F^{A_1}_1(u,v)=(f(u),g(u)v)$. Then we may consider $F_k^{A_1}\in C^2(\R^{2\times m},\R^2)$ for $2\leq k\leq d$ and $F_j^{A_2}\in C^2(\R^{2\times m},\R^2)$ for $1\leq j\leq d$ such that it satisfies \eqref{condition:planar} and the system admits a convex entropy.} For example, one may view system \eqref{sys:tri}  as a 1-D restriction of the following multi-dimensional triangular system:
	\begin{align}
	\pa_t u+\sum\limits_{i=1}^{d}\pa_if(u)&=0,\label{tri-sys-multiD1}\\
	\pa_t v_k+\sum\limits_{i=1}^{d}\pa_i(g(u)v_k)&=0,\mbox{ with }1\leq k\leq m,\label{tri-sys-multiD2}
	\end{align}
	for $(x,t)\in\R^d\times\R_+$. Note that system \eqref{tri-sys-multiD1}--\eqref{tri-sys-multiD2} inherits an entropy-entropy flux pair $(\eta,q)$ defined as follows,
	\begin{align*}
	\eta(u,v)&=\psi(u)+e^{-\phi(u)}K\left(ve^{\phi(u)}\right),\\
	q_i(u,v)&=P(u)+g(u)e^{-\phi(u)}K\left(ve^{\phi(u)}\right)\mbox{ where }P(u)=\int\limits_{0}^{u}\psi^{\p}(\si)f^{\p}(\si)\,d\si
	\end{align*}
	where $K:\R^m\rr\R,\psi:\R\rr\R$ are $C^2$ strictly convex functions and $\phi$ is the primitive of $\frac{g^{\p}(u)}{g(u)-f^{\p}(u)}$. We also assume that $u\mapsto e^{-\phi(u)}$ is concave. Note that
	\begin{equation*}
	D_u\eta(u,v)=\psi^{\p}(u)-e^{-\phi(u)}K(ve^{\phi(u)})\phi^{\p}(u)+DK\cdot v\phi^{\p}(u)\mbox{ and }D_v\eta(u,v)=DK(ve^{\phi(u)}).
	\end{equation*}
	Further, we have
	\begin{align*}
	D_{uu}\eta&=\psi^{\p\p}(u)+e^{-\phi(u)}(\phi^{\p}(u)^2-\phi^{\p\p}(u))K+(-\phi^{\p}(u)^2+\phi^{\p\p}(u))DK\cdot v+e^{\phi(u)}\phi^{\p}(u)^2v^TD^2Kv,\\
	D_{uv}\eta&=e^{\phi(u)}\phi^{\p}(u)D^2Kv\mbox{ and }D_{vv}\eta=e^{\phi(u)}D^2K.
	\end{align*}
	We wish to show that $D^2\eta$ is positive-definite. Since $D^2K$ is positive-definite, by Sylvester's criterion, it is enough to check that $\det(D^2\eta)>0$. Note that
	\begin{equation*}
	\det(D^2\eta)=\det\left(D_{uu}\eta D_{vv}\eta-D_{uv}\eta\otimes D_{uv}\eta\right).
	\end{equation*} 
	Let $\xi\in \R^{m}$ be any vector. We check that $\xi^T(D_{uu}\eta D_{vv}\eta-D_{uv}\eta\otimes D_{uv}\eta)\xi>0$ which proves that $\det(D^2\eta)>0$ and we can conclude that $D^2\eta$ is positive-definite. We compute that
	\begin{align*}
	\xi^T(D_{uu}D_{vv}-D_{uv}\otimes D_{uv})\xi&= e^{\phi(u)}\left(D_{uu}\eta\right)\xi^TD^2K\xi-\abs{\xi^TD^2Kv}^2e^{2\phi(u)}\phi^{\p}(u)^2\\
	&=e^{\phi(u)}\left[\psi^{\p\p}(u)+(\phi^{\p}(u)^2-\phi^{\p\p}(u))(Ke^{-\phi(u)}-DK\cdot v)\right]\xi^TD^2K\xi\\
	&+e^{2\phi(u)}\phi^{\p}(u)^2\left[\left(v^TD^2Kv\right)\left(\xi^TD^2K\xi\right)-\abs{\xi^TD^2Kv}^2\right].
	\end{align*}
	By the Cauchy-Schwartz inequality for the inner product induced by the symmetric, positive-definite matrix $D^2K$, we find that $v^TD^2Kv\xi^TD^2K\xi-\abs{\xi^TD^2Kv}^2\geq0$. Also, since $u\mapsto e^{-\phi(u)}$ is concave we get $\phi^{\p}(u)^2-\phi^{\p\p}(u)\leq0$ and since $K$ is convex with $K(0)=0$ it holds that
	\[
- K\left(ve^{\phi(u)}\right) + DK(ve^{\phi(u)})\cdot ve^{\phi(u)}\geq 0,
	\]
i.e. $K\left(ve^{\phi(u)}\right)e^{-\phi(u)}-DK(ve^{\phi(u)})\cdot v\leq 0$. Then, as $\psi,K$ are strictly convex, we indeed infer that $\det(D^2\eta)>0$ and $D^2\eta$ is positive-definite.
	


%
	
\section*{Acknowledgements} SSG and AJ acknowledge the support of the Department of Atomic Energy, Government of India, under project no. 12-R\&D-TFR-5.01-0520. SSG would also like to thank Inspire faculty-research grant DST/INSPIRE/04/2016/000237.


\section*{References}
	\begin{thebibliography}{99}
	
			\bibitem{AGV-exact} 
			\newblock Adimurthi, S. S. Ghoshal and  G. D. Veerappa Gowda,
			\newblock Exact controllability of scalar conservation laws with strict convex  flux.
			\newblock {\em Math. Control Relat. Fields}, 4, (4), 401--449, 2014.
			
	\bibitem{AGV-structure}
	\newblock Adimurthi,  S. S. Ghoshal and G. D. Veerappa Gowda, 
	\newblock Structure of entropy solutions to scalar conservation laws with strictly convex flux.
	\newblock {\em  J. Hyperbol. Differ. Eq.}, 4, 571--611, 2012.
	
				
	
			
		\bibitem{Boris} 
		\newblock B. Andreianov, C. Donadello, S. S. Ghoshal and U. Razafison,
		\newblock On the attainability set for triangular type system of conservation laws with initial data control.
		\newblock {\em J. Evol. Eq.}, 15, 3, 503--532, 2015.
		
	\bibitem{Baiti}
	\newblock P. Baiti and H. K. Jenssen,
	\newblock Blowup in $L^{\f}$ for a class of genuinely nonlinear hyperbolic systems of conservation laws.
	\newblock {\em Discrete Contin. Dynam. Systems}, 7, 4, 837--853, 2001.

	
\bibitem{Ball76}
\newblock J.M. Ball, 
\newblock Convexity conditions and existence theorems in nonlinear elasticity.
\newblock {\em Arch. Ration. Mech. Anal.}, 63, 4, 337--403, 1976.

\bibitem{Ball_open_problems}
\newblock J.M. Ball, 
\newblock Some open problems in elasticity.
\newblock {\em In Geometry, mechanics, and dynamics}, 3--59, 2002.
		
		\bibitem{BGSTW}
		\newblock C. Bardos, P. Gwiazda, A. \'Swierczewska-Gwiazda, E. S. Titi and E. Wiedemann,
		\newblock On the extension of Onsager's conjecture for general conservation laws.
		\newblock {\em J. Nonlinear Sci.}, 29, 2, 501--510, 2019.
		
\bibitem{BB}
\newblock S. Bianchini and A. Bressan,
\newblock Vanishing viscosity solutions of nonlinear hyperbolic systems.
\newblock{\em Ann. of Math. (2)} 161, no. 1, 223--342, 2005.
		
		
		\bibitem{BDS}
		\newblock Y. Brenier, C. De Lellis and L. Sz\'ekelyhidi Jr.,
		\newblock Weak-strong uniqueness for measure-valued solutions.
		\newblock{\em Comm. Math. Phys.}, 305, 2, 351--361, 2011.
		
\bibitem{Bressan-book}
\newblock A. Bressan,
\newblock Hyperbolic systems of conservation laws. The one-dimensional Cauchy problem.
\newblock{\em Oxford Lecture Series in Mathematics and its Applications, 20. Oxford University Press, Oxford}, 2000. xii+250 pp.
		

         \bibitem{westdickenberg}
	\newblock F. Cavalletti, M. Sedjro and M. Westdickenberg,
	\newblock A variational time discretization for compressible Euler equations.
	\newblock{\em Trans. Amer. Math. Soc.}, 371 (7), 5083--5155, 2019.

		\bibitem{ChHs}
		\newblock T. Chang and L. Hsiao,
		\newblock The Riemann Problem and Interaction of Waves in Gas Dynamics.
		\newblock Pitman Monographs and Surveys in Pure and Applied Mathematics, Vol. 41 (Longman Scientific \& Technical, Harlow, 1989).
		
		
		\bibitem{CF01}
		\newblock G.-Q. Chen and H. Frid,
		\newblock Uniqueness and asymptotic stability of Riemann solutions for the  compressible Euler equations.
		\newblock{\em Trans. Amer. Math. Soc.}, { 353} (3), 1103--1117 (electronic), 2001.
		
		\bibitem{ChFrLi}
		\newblock G.-Q. Chen, H. Frid and Y. Li,
		\newblock Uniqueness and stability of Riemann solutions with large oscillation in gas dynamics.
		\newblock{\em Comm. Math. Phys.}, { 228} (2), 201--217, 2002.
		
		\bibitem{CDK}
		\newblock E. Chiodaroli, C. De Lellis and O. Kreml,
		\newblock Global ill-posedness of the isentropic system of gas dynamics.
		\newblock{\em Comm. Pure Appl. Math.}, 68, 7, 1157--1190, 2015.
		
		
		\bibitem{Chiodaroli}
		\newblock	 E. Chiodaroli, O. Kreml, V. M{\'a}cha and S. Schwarzacher,
	\newblock Non-uniqueness of admissible weak solutions to the compressible {E}uler equations with smooth initial data.
	\newblock {arXiv preprint arXiv:1812.09917}, 2018.
	
	\bibitem{tzavaras_cleopatra}
	\newblock	C. Christoforou and A.E. Tzavaras,
	\newblock Relative entropy for hyperbolic--parabolic systems and application to the constitutive theory of thermoviscoelasticity.
	\newblock {\em Arch. Ration. Mech. Anal.}, { 229} (1), 1--52, 2018.

		
	\bibitem{CET}
	\newblock	P. Constantin, W. E and E. S. Titi,
	\newblock Onsager's conjecture on the energy conservation for solutions of {E}uler's equation.
	\newblock {\em Comm. Math. Phys.}, { 165} (1), 207--209, 1994.
	
	\bibitem{Dacorogna}
	\newblock	B. Dacorogna,
	\newblock Direct methods in the calculus of variations.
	\newblock Springer Science \& Business Media, vol. 78, 2007.
	
	\bibitem{Daf:Charac}
	\newblock C. M. Dafermos,
	\newblock Generalized characteristics and the structure of solutions of hyperbolic conservation laws.
	\newblock{\em Indiana Univ. Math. J.} 26, 6, 1097--1119, 1977.
	
	
	\bibitem{Daf79}
	\newblock C. M. Dafermos,
	\newblock The second law of thermodynamics and stability.
	\newblock{\em Arch. Ration. Mech. Anal.}, 70, 2, 167--179, 1979.
	
	\bibitem{Daf86}
	\newblock C. M. Dafermos,
	\newblock Quasilinear hyperbolic systems with involutions.
	\newblock{\em Arch. Ration. Mech. Anal.}, 94, 4, 373--389, 1986.
	
	\bibitem{Daf_book}
	\newblock C. M. Dafermos,
	\newblock Hyperbolic conservation laws in continuum physics. Second edition,
	\newblock Grundlehren der Mathematischen Wissenschaften [Fundamental Principles of Mathematical Sciences], 325, {\em Springer-Verlag, Berlin}, 2005, xx+626 pp.
	
	\bibitem{Deb}
	\newblock T. D{e}biec,
	\newblock On entropy conservation for general systems of conservation laws.
	\newblock {preprint}, arXiv:1910.05793, 2019.
	
	
	\bibitem{DeLS}
	\newblock C. De Lellis and L. Sz{\'e}kelyhidi Jr.,
	\newblock On admissibility criteria for weak solutions of the Euler equations.
	\newblock{\em Arch. Ration. Mech. Anal.}, {195} , 1, 225--260, 2010.
	
	
	{\color{black}
		\bibitem{tzavaras}
		\newblock S. Demoulini, D. M. A. Stuart and A. E. Tzavaras,
		\newblock	A variational approximation scheme for three-dimensional elastodynamics with polyconvex energy.
		\newblock{\em Arch. Ration. Mech. Anal.}, 157, 4, 325--344, 2001
	    
	    \bibitem{tzavaras-weak-strong}
	    \newblock S. Demoulini, D. M. A. Stuart, and A. E. Tzavaras,
	    \newblock Weak-strong uniqueness of dissipative measure-valued solutions for polyconvex elastodynamics.
	    \newblock{\em Arch. Ration. Mech. Anal.}, 205, 3, 927--961, 2012.
}
	
	{\color{black}
	\bibitem{Diperna_measure-valued}
	\newblock R. J. DiPerna,
	\newblock Measure-valued solutions to conservation laws.
	\newblock{\em Arch. Ration. Mech. Anal.}, {88}, 223--270, 1985.}
	
	\bibitem{DiPerna}
	\newblock R. J. DiPerna,
	\newblock Uniqueness of solutions to hyperbolic conservation laws.
	\newblock{\em Indiana Univ. Math. J.}, 28, 1, 137--188, 1979.
	
	
	
	
	\bibitem{FGGW}
	\newblock E. Feireisl, P. Gwiazda, A. \'Swierczewska-Gwiazda and E. Wiedemann,
	\newblock Regularity and energy conservation for the compressible {E}uler equations.
	\newblock {\em Arch. Ration. Mech. Anal.}, { 223}, 3, 1375--1395, 2017.
	
	
	\bibitem{FGJ}
	\newblock E. Feireisl, S. S. Ghoshal and A. Jana,
	\newblock On uniqueness of dissipative solutions to the isentropic	Euler system.
	\newblock{\em Comm. Partial Differential Equations}, 44, 12, 1285--1298, 2019.
	
	\bibitem{FeiKre}
	\newblock E. Feireisl and O. Kreml,
	\newblock Uniqueness of rarefaction waves in multidimensional compressible Euler system.
	\newblock {\em J. Hyperbolic Differ. Eq.}, 12, no. 3, 489--499, 2015.
	
	\bibitem{FeKeVa}
	\newblock E. Feireisl, O. Kreml, and A. Vasseur,
	\newblock Stability of the isentropic Riemann solutions of the full multidimensional Euler system.
	\newblock {\em SIAM J. Math. Anal.}, 47, no. 3, 2416--2425, 2015.
	
	\bibitem{GJ}
	\newblock S. S. Ghoshal and A. Jana,
	\newblock Uniqueness of dissipative solutions to the complete Euler system.
	\newblock {\em preprint}, arXiv:1905.06919, 2019.
	
	\bibitem{stochastic}
	\newblock S. S. Ghoshal, A. Jana and B. Sarkar,
	\newblock Uniqueness and energy balance for isentropic Euler equation with stochastic forcing.
	\newblock {\em preprint}, 2020.
	

	
{\color{black}	\bibitem{GKS}
	\newblock P. Gwiazda, O. Kreml and A. \'Swierczewska-Gwiazda,
	\newblock Dissipative measure-valued solutions for general conservation laws.
	\newblock{\em Ann. Inst. H. Poincare Anal. Non Lin\'eaire}, 37, 3, 683--707,
2020.
}
	
	\bibitem{GMS}
	\newblock P. Gwiazda, M. Mich\'alek,and A. \'Swierczewska-Gwiazda, 
	\newblock A note on weak solutions of conservation laws and energy/entropy conservation. 
	\newblock{\em Arch. Ration. Mech. Anal.}, 229, 3, 1223--1238, 2018.
	
	
	\bibitem{KS19}
	\newblock K. Koumatos and S. Spirito,
	\newblock Quasiconvex elastodynamics: weak-strong uniqueness for measure-valued solutions.
	\newblock{\em	Comm. Pure Appl. Math.}, 72, 6, 1288--1320, 2019.
	
		
	\bibitem{KV}
	\newblock K. Koumatos and A. Vikelis,
	\newblock $\mathcal{A}$-quasiconvexity, G\r{a}rding inequalities and applications in PDE constrained problems in dynamics and statics.
	\newblock {\em preprint}, arXiv:2005.12803, 2020. 
	

	\bibitem{LaTz}
      \newblock C. Lattanzio and A. E. Tzavaras,
      \newblock Structural properties of stress relaxation and convergence from viscoelasticity to polyconvex elastodynamics.
      \newblock{\em Arch. Ration. Mech. Anal.}, 180, no. 3, 449--492, 2006.
      

	
	\bibitem{vasseur}
	\newblock N. Leger and A. Vasseur,
	\newblock Relative entropy and the stability of shocks and contact discontinuities for systems of conservation laws with non-BV perturbations.
	\newblock{\em Arch. Ration. Mech. Anal.}, 201, 1, 271--302, 2011.
	
	
  \bibitem{Qin}
	\newblock T. Qin,
	\newblock Symmetrizing nonlinear elastodynamic system.
	\newblock{\em Journal of Elasticity}, 50, 3, 271--302, 245--252, 1998.
	
	
	\bibitem{Serre}
        \newblock D. Serre,
      \newblock Hyperbolicity of the nonlinear models of {M}axwell's equations.
      \newblock {\em Arch. Ration. Mech. Anal.} 172, 3, 309--331, 2004.
	

	
	\bibitem{Smo}
	\newblock J. Smoller,
	\newblock Shock waves and reaction-diffusion equations.
	\newblock Grundlehren der Mathematischen Wissenschaften [Fundamental Principles of Mathematical Science], 258, {\em Springer-Verlag, New York-Berlin}, xxi+581, 1983.
	

	\bibitem{Wid18}
	\newblock E. Wiedemann,
	\newblock Weak-strong uniqueness in fluid dynamics.
	\newblock Partial differential equations in fluid mechanics, 289--326, {\em London Math. Soc. Lecture Note Ser.}, 452,{\em Cambridge Univ. Press, Cambridge}, 2018.
	
	
	\end{thebibliography}
\end{document}